\documentclass[10pt]{amsart}
\usepackage{amssymb,url}
\usepackage[utf8]{inputenc}
\usepackage{a4wide}
\usepackage{amsmath, amssymb, bbm}
\usepackage{amsthm}	
\usepackage{mathtools}
\usepackage{bbm}
\usepackage{blkarray}
\usepackage[matrix,arrow,curve]{xy}
\usepackage{tikz}
\usepackage{tikz-cd}
\usepackage{graphicx}
\usepackage{color}
\usepackage{enumerate}
\definecolor{darkblue}{rgb}{0,0,0.6}
\usepackage[ocgcolorlinks,colorlinks=true, citecolor=darkblue, filecolor=darkblue, linkcolor=darkblue, urlcolor=darkblue]{hyperref}
\usepackage[capitalize,noabbrev]{cleveref}
\usepackage[mathscr]{euscript}
\usepackage{comment}
\usepackage{pdfpages}
\usepackage{stmaryrd}
\usepackage{lmodern}

\newtheorem{Thm}[subsection]{Theorem}
\newtheorem{Prop}[subsection]{Proposition}
\newtheorem{Cor}[subsection]{Corollary}
\newtheorem{Lemma}[subsection]{Lemma}

\theoremstyle{remark}

\newtheorem{Ex}[subsection]{Example}

\theoremstyle{definition}
\newtheorem{Def}[subsection]{Definition}
\newtheorem*{ackn}{Acknowledgements}
\newtheorem*{Notation}{Notation}
\newtheorem*{Warning}{Warning}
\newtheorem*{Theorem}{Theorem}
\newtheorem{Rmk}[subsection]{Remark}

\theoremstyle{defs2}

\newcommand{\Z}{\mathbb{Z}}
\renewcommand{\L}{\mathscr{L}}
\newcommand{\m}{\mathscr{M}}
\newcommand{\M}{\mathrm{M}}
\newcommand{\lto}{\longrightarrow}
\newcommand{\SF}{\mathrm{SF}}
\newcommand{\BG}{\mathrm{BG}}
\newcommand{\BO}{\mathrm{BO}}
\newcommand{\BPL}{\mathrm{BPL}}
\newcommand{\BTop}{\mathrm{BTop}}
\newcommand{\BSG}{\mathrm{BSG}}
\newcommand{\BSO}{\mathrm{BSO}}

\newcommand{\B}{\mathrm{B}}
\newcommand{\G}{\mathrm{G}}
\newcommand{\PL}{\mathrm{PL}}
\newcommand{\Top}{\mathrm{Top}}
\renewcommand{\o}{\mathrm{O}}
\newcommand{\Hom}{\mathrm{Hom}}
\newcommand{\Sq}{\mathrm{Sq}}
\renewcommand{\l}{\mathrm{L}}

\renewcommand{\H}{\mathrm{H}}
\newcommand{\E}{\mathbb{E}}
\renewcommand{\S}{\mathbb{S}}
\newcommand{\gl}{\mathrm{gl}}

\newcommand{\MSO}{\mathrm{MSO}}
\newcommand{\MSG}{\mathrm{MSG}}
\newcommand{\MSpin}{\mathrm{MSpin}}
\newcommand{\MSGpin}{\mathrm{MSGpin}}
\newcommand{\SO}{\mathrm{SO}}
\newcommand{\Spin}{\mathrm{Spin}}

\renewcommand{\S}{\mathbb{S}}

\newcommand{\Sp}{\mathrm{Sp}}
\newcommand{\F}{\mathscr{F}}
\newcommand{\g}{\mathscr{G}}
\newcommand{\Pic}{\mathrm{Pic}}

\newcommand{\map}{\mathrm{map}}
\newcommand{\Map}{\mathrm{Map}}

\newcommand{\RP}{\mathbb{RP}}

\newcommand{\ks}{\mathrm{ks}}

\newcommand{\Spc}{\mathrm{An}}

\newcommand{\Fun}{\mathrm{Fun}}

\newcommand{\Loc}{\mathscr{L}}
\newcommand{\Orloc}{\mathscr{O}}
\newcommand{\op}{\mathrm{op}}
\newcommand{\id}{\mathrm{id}}
\newcommand{\D}{\mathscr{D}}
\newcommand{\R}{\mathbb{R}}
\renewcommand{\g}{\mathscr{G}}
\newcommand{\Tw}{\mathrm{Tw}}
\newcommand{\Th}{\mathrm{Th}}

\newcommand{\h}{\mathscr{h}}

\newcommand{\CP}{\mathbb{CP}}
\newcommand{\C}{\mathscr{C}}
\newcommand{\fib}{\mathrm{fib}}
\newcommand{\cE}{\mathscr{E}}
\renewcommand{\h}{\mathscr{H}}

\DeclareMathOperator*{\colim}{colim}

\begin{document}

\keywords{Poincar\'e duality spaces, Spivak normal fibration, reducibility}
\subjclass[2010]{55R25, 57P10}

\title{Reducibility of low dimensional Poincar\'e duality spaces}
\date{\today}

\author{Markus Land}
\address{Department of Mathematical Sciences, University of Copenhagen, 2100 Copenhagen, Denmark}
\email{markus.land@math.ku.dk}

\thanks{I am happy to acknowledge the support provided by the CRC 1085 ``Higher invariants'' at the University of Regensburg granted by the DFG and by the Danish National Research Foundation through the Copenhagen Centre for Geometry and Topology (DNRF151).}

\begin{abstract}
We discuss Poincar\'e duality complexes $X$ and the question whether or not their Spivak normal fibration admits a reduction to a vector bundle in the case where the dimension of $X$ is at most 4. We show that in dimensions less than 4 such a reduction always exists, and in dimension 4 such a reduction exists provided $X$ is orientable. In the non-orientable case there are counterexamples to reducibility by Hambleton--Milgram. 
\end{abstract}

\maketitle

\tableofcontents

The purpose of this paper is to collect what is known about the question whether or not a Poincar\'e duality space is reducible, i.e.\ whether its Spivak normal fibration admits a reduction to a vector bundle, focussing on low dimensions. The main result of this paper is the following.

\begin{Theorem}
Let $X$ be a Poincar\'e duality space of dimension $\leq 3$ or an oriented Poincar\'e duality space of dimension 4. Then $X$ is reducible, i.e.\ its Spivak normal fibration admits a vector bundle reduction.
\end{Theorem}

The conclusion of the 4-dimensional part of the theorem is claimed in \cite[page 95 Example (1)]{Spivak} without reference or any indication of a proof. 
The main purpose of this paper, when first written, was to fill this gap. Unfortunately, the first version of this paper contained a gap which I will explain in \cref{Warning:error}. 
When Diarmuid Crowley made me aware of this gap, I had worked out another proof using bordism and L-theoretic arguments which was suggested to me by Larry Taylor. I wish to thank Larry Taylor for encouraging me to write up the argument in this note and apologise for the delay in preparing this manuscript. 

In dimension 2, the above result is a consequence of work of Wall's. In dimension 3, to the best of my knowledge it was first shown in \cite{Cannon-conjecture} and I have since then seen that the result is also stated in \cite{Hillman} without reference. After the incorrect argument of mine appeared, Hambleton gave a proof of the 4-dimensional case of the above theorem in \cite{Hambleton}, based on earlier work of Hambleton--Milgram \cite{HM}. In summary, the main results of this paper are already in the literature, the main purpose of this note is to give an overview of different methods that are available, and to give an alternative proof to Hambleton's in the 4-dimensional case. 
\newline

\subsection*{Motivation -- Wall's Conjecture}
Let $\Gamma$ be an $n$-dimensional Poincar\'e duality group, or PD(n) group for short, i.e.\ the fundamental group of an aspherical PD space $X$ of dimension $n$. Then Wall conjectured that $X$ is homotopy equivalent to a (necessarily aspherical) manifold. In particular, it is conjectured that $X$ is reducible. In contrast to reducibility of general PD spaces, the reducibility of aspherical PD spaces of dimension at least 5 is implied by the Farrell--Jones conjectures in K- and L-theory (for the given group). If they hold, then there is an ANR homology manifold homotopy equivalent to the aspherical PD space in question, as follows from the 4-periodic version of the total surgery obstruction of Ranicki, see \cite[Chapter 25]{Ranicki}. Indeed, the Farrell--Jones conjectures imply that the obstruction group for the existence of such an ANR homology manifold vanishes. Furthermore, \cite[Theorem 16.6]{FP} states that ANR homology manifolds have reducible Spivak fibration. In fact, this was used by Bartels--L\"uck--Weinberger to prove amongst other things that Wall's conjecture is true for hyperbolic groups whose boundary is a sphere (such groups are Poincar\'e duality groups by \cite{BM}), provided the group in question has dimension at least 6, see \cite{BLW}. 

There exist several papers which aim to prove theorems about 3-dimensional PD spaces which are known to hold for 3-manifolds. We first learned of such results and Wall's conjecture in a nice talk of Boileau at the conference ``Manifolds and Groups'' in Regensburg in 2017, and then asked whether it was at least known that every 3-dimensional PD space is reducible. To our surprise, this was not documented in the literature. Soon after, Wolfgang L\"uck found a proof of this statement using the geometric approach we outline in \cref{section3}; in fact, the approach presented there is based on L\"ucks proof of the 3-dimensional case which was used and written up in work of Cappell--L\"uck--Weinberger on a stable version of Cannon's conjecture \cite{Cannon-conjecture}. In discussions with L\"uck, we started talking about the 4-dimensional case, which is how this paper came to life.

\subsubsection*{Organisation of the paper}
In \cref{section1}, we recall Poincar\'e duality spaces and all ingredients necessary to understand the statement of the main theorem. The obstruction theory for vector bundle reductions of spherical fibrations is set up and direct consequences are established. \cref{section2} is devoted to the proof of the main theorem. In \cref{section3}, we rephrase the reducibility question in a more geometric fashion, and explain a different approach to the main theorem (which allows to prove the main result in the case of spin PD spaces).
Finally, in \cref{section4}, we note that the orientability condition in the main theorem cannot be dropped using an example of Hambleton--Milgram and briefly discuss the failure of reducibility in dimension 5 in terms of the geometric method of \cref{section3}.
In the appendix, we explain how Poincar\'e duality spaces can be characterised from a parametrised homotopy theoretic point of view. While surely well-known to homotopy theorists, we thought it might be helpful to add this material here. 
\newline

\begin{ackn}
I thank Wolfgang L\"uck for sharing his insights with me and making me aware of the question. I thank Ian Hambleton for very helpful comments, Larry Taylor and Andrew Ranicki for nice discussions about the topic of this paper, and Diarmuid Crowley for spotting a mistake in an earlier version and discussions about the topic. I also thank Ulrich Bunke, Fabian Hebestreit, and Christoph Schrade for comments on the first version of this paper, and Thomas Nikolaus for discussions about PD complexes in general. Finally, I thank Michael Weiss for very helpful suggestions.
\end{ackn}

\section{Poincar\'e duality spaces}\label{section1}

\subsection{Poincar\'e duality spaces}\label{Section:PD-spaces-classical}
Roughly speaking, a space is called a Poincar\'e duality space if its (co)homology satisfies Poincar\'e duality. 
More precisely, a Poincar\'e duality space of dimension $n$ 
is a space $X$ of the homotopy type of a finite CW complex, such that there exists a pair $(\L,[X])$ where
\begin{enumerate}
\item[-] $\L$ is a local system on $X$ which is point-wise infinite cyclic, called the orientation local system, and 
\item[-] $[X]$ is an element in $\H_n(X;\L)$, called the fundamental class,
\end{enumerate}
satisfying the property that the map 
\[ \H^k(X;\m) \xrightarrow{- \cap [X]} \H_{n-k}(X;\m\otimes\L) \]
is an isomorphism for any local system $\m$ on $X$ and all $k \in \Z$. In particular, the (co)homology of $X$ vanishes in degrees larger than the dimension of $X$. A Poincar\'e duality space equipped with such a pair $(\L,[X])$ will be called $\L$-oriented. From now on, we shall restrict to connected Poincar\'e duality spaces.
Choosing a base point $x_0$ of $X$, the orientation local system $\L$ induces a homomorphism 
\[w_1(X)\colon \pi_1(X,x_0) \lto \mathrm{Aut}(\L_{x_0}) \cong \Z/2\] 
called the induced orientation character and 
the Poincar\'e duality space $X$ is called orientable if $w_1(X) = 0$. This is equivalent to the condition that there exists an isomorphism $\varphi \colon \L \to \Z$ from the orientation local system to the constant local system. For an $\L$-oriented Poincar\'e duality space $X$, the choice of such an isomorphism provides a $\Z$-orientation, and then $X$ is simply called oriented.
In this case, Poincar\'e duality holds with constant coefficients. Reducing coefficients modulo 2, one sees that $\L$ becomes a constant local system with values $\Z/2$, so that $X$ is canonically $\Z/2$-oriented. Thus there is always untwisted Poincar\'e duality with $\Z/2$ coefficients. From here on, we will refer to Poincar\'e duality spaces as PD spaces or PD complexes and sometimes leave an $\L$-orientation implicit.

It is a theorem of Spivak, see \cite{Spivak}, that any PD space $X$ carries a canonical stable spherical fibration, the Spivak normal fibration, whose classifying map we will denote by $\SF(X) \colon X \to \BG$, see \cref{appendix} for a discussion of the Spivak normal fibration from a parametrised homotopy theoretic point of view. Here $\BG$ denotes a classifying space for stable spherical fibrations (of formal dimension $0$). Let us denote by $\M(\xi)$ the Thom spectrum of a stable spherical fibration $\xi$. Then the Spivak normal fibration comes with a \emph{stable collapse map} $\S^n \to \M(\SF(X))$, where $n$ is the dimension of $X$ and $\S^n$ is the $n$-fold suspension of the sphere spectrum $\S$, such that the following holds: The canonical map
\[ \pi_n(\M(\SF(X))) \lto \H_n(\M(\SF(X));\Z) \stackrel{\cong}{\lto} \H_n(X;\L) \]
sends the collapse map to a fundamental class of $X$. In particular, the orientation local system associated to the Spivak normal fibration is identified with $\L$ in order to obtain the last (twisted) Thom isomorphism map.

In fact, Spivak in addition showed that the Spivak normal fibration is uniquely determined in the following sense: If $\xi$ is a spherical fibration over $X$, also equipped with a map $c\colon\S^n \to \M(\xi)$, such that the composite 
\[ \pi_n(\M(\xi)) \lto \H_n(\M(\xi)) \stackrel{\cong}{\lto} \H_n(X;\L) \]
sends the map $c$ to the fundamental class of $X$, then $\xi$ is equivalent to $\SF(X)$, and the equivalence can be chosen to send the collapse map of $\xi$ to that of $\SF(X)$. Even more is true, namely Wall showed that, in the above described situation, there exists a \emph{unique} such equivalence. See \cite[Corollary 3.4 \& Theorem 3.5]{Wall} and \cite[Sections 3.4 \& 10.3]{CLM} for more details and the appendix for a different perspective on this matter. This \emph{universal property} of the Spivak normal fibration can be used to show that closed manifolds are PD spaces and that their Spivak normal fibration is the spherical fibration underlying the stable normal bundle.
The universal property of the Spivak normal fibration will also be revisited in \cref{appendix}.

\begin{Ex}
If $M$ is a closed manifold, then $M$ is an $\hat{\L}$-oriented PD space: Here $\hat{\L}$ refers to the local system induced by the orientation double cover $\hat{M}$ via the bundle of free abelian groups of rank one given by $\hat{M}\times_{\Z/2} \Z$. The fact that the orientation double cover of any closed manifold is canonically oriented implies that $M$ is $\hat{\L}$-oriented. The Spivak normal fibration of $M$ is given by the stable spherical fibration underlying the stable normal bundle $\nu(M)$ of $M$: This follows both from Spivak's construction and also the uniqueness of the Spivak normal fibration: By choosing an embedding $i\colon M \subseteq S^{n+k}$, there is an induced (unstable) collapse map $S^{n+k} \to \mathrm{Th}(\nu(i))$ to the Thom \emph{space} of the normal bundle of this embedding. This unstable collapse map induces a stable collapse map from $\S^n$ to the Thom spectrum of the stable normal bundle. 
\end{Ex}

For a PD space $X$, we will write $w(X) = w(\SF(X))^{-1}$ for the total Stiefel--Whitney class of the inverse of the Spivak normal fibration. For a closed manifold $M$ we get that 
\[ w(M) = w(\SF(M))^{-1} = w(\nu(M))^{-1} = w(\tau(M)) \]
where $\tau(M)$ denotes the tangent bundle of $M$. The definition of Stiefel--Whitney classes for PD spaces is thus a generalisation of the one for manifolds.

\subsection{The Wu--formula}
Let us now recall the Wu classes of an $n$-dimensional PD space $X$.
By Poincar\'e duality with $\Z/2$ coefficients, the canonical map 
\[ \H^k(X;\Z/2) \lto \Hom(\H^{n-k}(X;\Z/2),\Z/2),\]
given by sending $x$ to the homomorphism sending $y$ to $\langle x\cup y, [X] \rangle$ is an isomorphism. The right hand side contains the homomorphism
\[ y \longmapsto \langle \Sq^k(y),[X] \rangle.\]
Thus, for every $k \geq 0$ there is a unique class $v_k(M)$ representing this homomorphism. We denote by $v(M)$ the total Wu class. It is characterised by the equation
\[ \langle \Sq(x),[X] \rangle = \langle v(M)\cup x,[X] \rangle \]
for all $x \in \H^*(X;\Z/2)$. Notice that $v_k(M) = 0$ once $2k$ is larger than $n$.

In fact, one can (and should) define Wu classes for spherical fibrations as follows.
For a spherical fibration $\xi\colon E \to B$, let $\Phi$ denote the Thom isomorphism on mod 2 cohomology. Then one defines 
\[ v_i(\xi) = \Phi^{-1}\big(\chi(\Sq^i)(u)\big) \in \H^i(B;\Z/2)\]
where $u$ is a mod 2 Thom class of $\xi$ and $\chi$ is the antipode of the mod 2 Steenrod algebra. With this definition it is formal to see, using $\chi(\Sq) = \Sq^{-1}$, that there is the equation
\[ \Sq(v(\xi)) = w(\xi)^{-1}.\]
In the universal case, i.e.\ in the cohomology of $\BG$, the element $w(\xi)$ is not invertible, but it becomes invertible when we complete the cohomology at the ideal of positively graded elements. Furthermore, the ordinary mod 2 cohomology injects into the completed cohomology: it is the inclusion of the polynomial ring in the power series ring. For finite CW complexes, the cohomology is complete, so that $w(\xi)$ indeed becomes invertible when pulled back to any finite CW complex.

The following proposition is then known as a Wu--formula, see \cite[Proposition III.3.6]{Browder} for a proof.
\begin{Prop}
Let $X$ be a PD space. Then $v(\SF(X)) = v(X)$ and in particular one finds $\Sq(v(X)) = w(X)$.
\end{Prop}

\subsection{Detecting the Spivak fibration and Reducibility}
We start with a definition. Let $X$ and $Y$ be $\L_X$- and $\L_Y$-oriented PD spaces of dimension $n$. 
\begin{Def}\label{degree 1 maps}
A degree 1 map from $Y$ and $X$ consists of a pair $(f,\alpha)$, where 
\begin{enumerate}
\item[-] $f\colon Y \to X$ is a map, and 
\item[-] $\alpha \colon \L_Y \to f^*(\L_X)$ is an isomorphism,
\end{enumerate}
satisfying the property that the map induced by the pair $(f,\alpha)$
\[ \H_n(Y;\L_Y) \lto \H_n(X;\L_X) \]
sends the fundamental class of $Y$ to that of $X$.
\end{Def}

\begin{Rmk}
Let $X$ and $Y$ be \emph{oriented} PD spaces of dimension $n$. In this case, the isomorphisms $\L_X \cong \Z$ and $\L_Y\cong \Z$, i.e.\ the orientations of $X$ and $Y$, provide a canonical choice for the isomorphism $\alpha$, namely such that $\alpha$ corresponds to the identity of $\Z$. Choosing this isomorphism, a degree 1 map between oriented PD spaces is exactly what one is used to. However, by definition, one may also choose the other isomorphism between $\L_Y$ and $f^*(\L_X)$ that corresponds to multiplication by $-1$, and then degree 1 maps in the sense of \cref{degree 1 maps} correspond to what one usually calls maps of degree $-1$.
\end{Rmk}

We will make use of the following well-known lemma which gives a method of detecting that some spherical fibration over a PD space is in fact the Spivak normal fibration. 

\begin{Lemma}\label{detecting the SF}
Suppose given a pullback diagram
\[\begin{tikzcd} \SF(Y) \ar[r, "\bar{f}"] \ar[d] & \xi \ar[d] \\ Y \ar[r,"f"] & X \end{tikzcd} \]
where $X$ and $Y$ are $n$-dimensional $\L_X$- and $\L_Y$-oriented PD spaces and where $f$ is of degree 1. If $w_1(\xi) = w_1(X)$, then 
the composition 
\[ \S^n \lto \M(\SF(Y)) \xrightarrow{\M(\bar{f})} \M(\xi) \]
exhibits $\xi$ as a Spivak normal fibration of $X$.
\end{Lemma}
\begin{proof}
One simply checks the universal property, i.e.\ that the class corresponding to $[X]$ in $\H_n(\M(\xi);\Z)$ under the Thom isomorphism comes from a stable map $\S^n \to \M(\xi)$ under the Hurewicz homomorphism. Notice that one uses the twisted Thom isomorphism
\[ \H_n(\M(\xi);\Z) \cong \H_n(X;\L_X)\]
and it is here that we use the assumption $w_1(\xi) = w_1(X)$. As indicated in the statement of the lemma, the map $\bar{f}$ induces a map on Thom spectra $\M(\SF(Y)) \to \M(\xi)$ and so one considers the composite
\[ \S^n \lto \M(\SF(Y)) \lto \M(\xi).\]
It has the desired properties because $f$ has degree 1 and the diagram
\[ \begin{tikzcd} \H_k(\M(\SF(Y));\Z) \ar[r,"\M(\bar{f})_*"] \ar[d,"\cong"] & \H_k(\M(\xi);\Z) \ar[d,"\cong"] \\ \H_k(Y;\L_Y) \ar[r,"f_*"] & \H_k(X;\L_X) \end{tikzcd}\]
commutes.
\end{proof}

Recall that for a manifold $M$, the Spivak normal fibration is given by the spherical fibration underlying the stable normal bundle of $M$, i.e.\ the diagram
\[ \begin{tikzcd} & \BO \ar[d,"J"] \\ M \ar[r,"\SF(M)"'] \ar[ur, bend left, "\nu(M)"] & \BG \end{tikzcd}\]
commutes up to homotopy. This motivates the following definition.

\begin{Def}
Let $X$ be a PD space. $X$ is said to be \emph{reducible} if the Spivak normal fibration admits a reduction to a vector bundle, i.e.\ if a dashed arrow in the diagram
\[ \begin{tikzcd} & \BO \ar[d,"J"] \\ X \ar[r,"\SF(X)"'] \ar[ur, bend left, dashed] & \BG \end{tikzcd}\]
exists making the diagram commute up to homotopy.
\end{Def}

\begin{Rmk}\label{rem:different-reducibilities}
By the chain of maps
\[ \BO \lto \BPL \lto \BTop \lto \BG,\]
induced by the evident inclusions,
there are three things that could potentially be meant by reducibility, namely admitting a lift to either of the three spaces $\BO$, $\BPL$, or $\BTop$. We will call these reducible, PL-reducible, and Top-reducible, respectively.
We will only focus on the reducibility question (i.e.\ admitting a lift to $\BO$) in this paper since for PD complexes of dimension at most 4, the potentially different meanings collapse to the same. In general, this is not the case when $X$ has non-trivial cohomology in degrees larger than four, and one should always be more specific in which sense reducibility is meant. We will briefly come back to this in \cref{section4}.
\end{Rmk}

\begin{Ex}\label{Example:SF}
Let $X$ be a PD space which is homotopy equivalent to a closed smooth manifold $M$. Then $X$ is reducible. This follows from \cref{detecting the SF}: Let $f \colon M \to X$ be a homotopy equivalence, and let $\xi = (f^{-1})^*(\nu(M))$. Then one has a diagram as in \cref{detecting the SF}, hence the claim follows.
\end{Ex}

Thus reducibility is an obstruction to the question whether a given PD space is homotopy equivalent to a closed manifold. In fact, it is the first obstruction in the surgery program. In some cases, it is the only obstruction. We will briefly come back to this in \cref{subsection-surgery}
\newline

We now recall that by work of Boardman--Vogt, the $J$-homomorphism $\BO \to \BG$, which classifies the underlying spherical fibration of the universal stable vector bundle, is a map of (connected and hence group-like) $\E_\infty$-spaces, i.e.\ of infinite loop spaces. Its fibre is denoted by $\G/\o$, and by a result of Sullivan's its associated cohomology theory $\G/\o^*(-)$ when evaluated on manifolds is that of bordism classes of smooth degree 1 normal maps. Hence it plays a prominent role in surgery theory. There is then a canonical fibre sequence
\[ \BO \stackrel{J}{\lto} \BG \lto \B(\G/\o)\]
and thus the obstruction to finding a vector bundle lift of a given spherical fibration $X \to \BG$ is given by the composite
\[ [ X \lto \BG \lto \B(\G/\o) ] \in \G/\o^1(X).\]

If $X$ has trivial cohomology above degree $n$, as is the case for an $n$-dimensional PD space, the canonical truncation map of $\B(\G/\o)$ 
induces an equivalence on homotopy classes of maps
\[ [X,\B(\G/\o)] \stackrel{\cong}{\lto} [X,\tau_{\leq n}(\B(\G/\o))]\]
by obstruction theory. Recall that for any space $Y$, the map $Y \to \tau_{\leq n}(Y)$ is the initial map to a space whose homotopy groups vanish in degrees above $n$.
The low dimensional homotopy groups of $\B(\G/\o)$ are well-known, we shall record the calculation as the following lemma.
\begin{Lemma}
There is a unique equivalence $\tau_{\leq 4}(\B(\G/\o)) \simeq K(\Z/2,3)$ as pointed, in fact as $\E_\infty$-spaces.
\end{Lemma}
\begin{proof}
It suffices to show that the homotopy groups of both sides agree, as then there exists such an equivalence. This equivalence is unique because the space of pointed, in fact of $\E_\infty$ self-equivalences of $K(\Z/2,3)$ is contractible.
To see the claim about homotopy groups, one considers the long exact sequence in homotopy groups associated to the fibration 
\[ \BO \stackrel{J}{\lto} \BG \lto \B(\G/\o).\]
The J-homomorphism induces an isomorphism on $\pi_i$ for $i=1,2$ and a surjection on $\pi_4$. Furthermore $\pi_3(\BO) = 0$ and $\pi_3(\BG) \cong \Z/2$.
\end{proof}

\begin{Lemma}\label{e_1}
The homotopy class of the composite 
\[ \BG \lto \B(\G/\o) \lto \tau_{\leq 4}(\B(\G/\o)) \simeq K(\Z/2,3) \] 
is the unique non-trivial class in
\[\ker\big(\H^3(\BG;\Z/2) \lto \H^3(\BO;\Z/2) \big).\]
\end{Lemma}
\begin{proof}
This follows from the Serre spectral sequence associated to the fibration
\[\BO \lto \BG \lto \B(\G/\o) \]
and the fact that the map $\H^*(\BG;\Z/2) \to \H^*(\BO;\Z/2)$ is surjective.
\end{proof}

\begin{Rmk}
In \cite{GitlerStasheff} a cohomology class $e_1 \in \H^3(\BG;\Z/2)$ is constructed by means of a secondary operation. It is called \emph{the first exotic class} and it is shown that 
\begin{enumerate}
\item[$\bullet$] $e_1$ is non-trivial, see \cite[Theorem 5.1]{GitlerStasheff}, and
\item[$\bullet$] $e_1$ lies in the kernel of the map $\H^3(\BG;\Z/2) \to \H^3(\BO;\Z/2)$, see \cite[before Theorem 5.2]{GitlerStasheff}.
\end{enumerate}
\cref{e_1} thus shows that the homotopy class of the composite
\[ \BG \lto \B(\G/\o) \lto K(\Z/2,3) \]
is given by $e_1$.
\end{Rmk}

\begin{Cor}
Let $X$ be a PD space of dimension at most 4. Then the obstruction to reducibility is given by
\[ e_1(\SF(X)) \in \H^3(X;\Z/2).\]
\end{Cor}

\begin{Cor}
Let $X$ be a PD space of dimension at most $2$. Then $X$ is reducible.
\end{Cor}

\begin{Rmk}\label{Rmk:2d-PD-spaces}
It turns out that even more is true, namely that every 2-dimensional PD space is in fact homotopy equivalent to a 2-dimensional manifold:
Wall showed that a 2-dimensional PD space with finite fundamental group is homotopy equivalent to $S^2$ or $\RP^2$, see \cite[Theorem 4.2 (iv) \& Theorem 4.3]{Wall}. If, on the other hand, $X$ has infinite fundamental group, then Wall showed that $X$ is aspherical, see \cite[Theorem 4.2 (vi)]{Wall}. To see that every 2-dimensional PD space is homotopy equivalent to a 2-dimensional manifold it hence suffices to know that every group satisfying 2-dimensional Poincar\'e duality is isomorphic to a surface group. This is the main result of \cite{EL}, building on earlier work of \cite{EM}.

For completeness, we add here that also every connected 1-dimensional PD space is homotopy equivalent to $S^1$ \cite[Theorem 4.2 (ii)]{Wall}.
\end{Rmk}

\begin{Cor}\label{reducibility for pi1 trivial 4 PD spaces}
Let $X$ be a 4-dimensional PD space such that $\H_1(X;\Z/2) = 0$, e.g.\ let $X$ be simply connected. Then $X$ is reducible.
\end{Cor}
\begin{proof}
By Poincar\'e duality we have $\H^3(X;\Z/2) \cong \H_1(X;\Z/2) = 0$, thus the obstruction group vanishes.
\end{proof}

\begin{Rmk}
\cref{reducibility for pi1 trivial 4 PD spaces} implies that the Spivak normal fibration of a simply connected topological 4-manifold admits a vector bundle reduction. This might sound surprising at first glance in light of the non-triviality of the Kirby--Siebenmann invariant: Recall that a topological manifold $M$ has a topological tangent bundle, whose classifying map is a map
\[ \tau(M) \colon M \lto \BTop.\]
A first obstruction for $M$ to admit a PL-structure or a smooth structure is that this classifying map must lift to $\BPL$. The obstruction to such a lift is given by the composite
\[ M \lto \BTop \lto \B(\Top/\PL).\]
Results of Kirby--Siebenmann say that $\B(\Top/\PL) \simeq K(\Z/2,4)$, see \cite{KS}, and so the obstruction of finding a PL-structure on the topological tangent bundle is given by an element
\[ \ks(M) \in \H^4(M;\Z/2)\]
called the Kirby--Siebenmann invariant. In dimensions less or equal to 6, finding a lift of $\tau(M)$ to $\BPL$ or $\BO$ is equivalent, as the space $\PL/\o$ is 6-connected by \cite{Moise, Cerf, KM}.

A prominent example of a simply connected topological $4$-manifold with non-trivial Kirby--Siebenmann invariant is given by Freedman's $E_8$-manifold. It follows that $\nu(E_8)$, the inverse of $\tau(E_8)$, also does not admit a vector bundle reduction. But by \cref{reducibility for pi1 trivial 4 PD spaces}, the underlying spherical fibration of $\nu(E_8)$, which is $\SF(E_8)$, \emph{does} admit a reduction to a vector bundle, say $V(E_8)$. It follows that $V(E_8)$ is not isomorphic to $\nu(E_8)$ as topological bundles because their Kirby--Siebenmann invariants are different.
\end{Rmk}

\section{The main theorem}\label{section2}
For the proof, we will make use of algebraic L-theory as defined and developed by Ranicki in e.g.\ \cite{Ranicki}. 
\begin{Notation}
We will denote quadratic L-theory of a ring $R$ by $\l^q(R)$, symmetric L-theory by $\l^s(R)$ and normal L-theory by $\l^n(R)$. Here we always mean the respective spectra, recall that $\l^n(R)$ is by definition the cofibre of the symmetrisation map $\l^q(R) \to \l^s(R)$. The L-groups will be denoted by 
\[ \l^q_n(R) = \pi_n(\l^q(R)) \] 
and likewise for the symmetric and normal case. Notice that this notation is different from the one used by Ranicki: He writes $\l^k(R)$ for what we denote $\l^s_k(R)$, $\l_k(R)$ for what we denote $\l^q_k(R)$ and $\widehat{\l}^k(R)$ for what we denote by $\l^n_k(R)$. In particular there is the unfortunate problem that $\l^n(R)$ in our notation denotes the \emph{normal} L-\emph{spectrum}, whereas in Ranicki's notation this would be the $n^{th}$ \emph{symmetric} L-\emph{group} of $R$. Furthermore, my notation $\l^s(R)$ (where the s stands for symmetric, of course) should not be confused with the superscript $s$ denoting the \emph{simple} decoration. We apologise for this notational inconvenience.
The upshot to take away is that when we write superscripts, they indicate which of the L-theories (quadratic, symmetric, normal) is meant, whereas subscripts refer to the homotopy groups of the corresponding spectrum. There will be no source of confusion in terms of decorations, as we will always use the free (or in our case equivalently the projective) decoration.
\end{Notation}

The starting point for the approach we want to take here is what we refer to as the Poincar\'e obstruction sequence, also called the Levitt--Jones--Quinn Poincar\'e bordism sequence in \cite{KT}: For any space $X$, there is an exact sequence (provided $n$ is at least 5)
\[ \dots \lto \mathrm{L}^q_n(\Z[\pi_1(X)]) \lto \Omega^\mathrm{P}_n(X) \lto \Omega^\mathrm{N}_n(X) \stackrel{\vartheta}{\lto} \mathrm{L}^q_{n-1}(\Z[\pi_1(X)]) \lto \dots \]
where $\Omega^\mathrm{P}_*(X)$ denotes oriented Poincar\'e bordism of $X$, $\Omega^\mathrm{N}_*(X) \cong \pi_*(X_+\otimes\MSG)$ is oriented normal bordism, and $\mathrm{L}^q_n(\Z[\pi_1(X)])$ denotes the $n^{th}$ quadratic L-group of $\Z[\pi_1(X)]$ (with free decoration, to be precise). See \cite[Proposition 19.6]{Ranicki} for this long exact sequence, which was first proven in \cite{Quinn, Levitt, Jones}, and further discussed in \cite{HV}. We emphasize at this point that we do not make use of the exactness of this sequence.
The map 
\[  \mathrm{L}^q_n(\Z[\pi_1(X)]) \lto \Omega^\mathrm{P}_n(X) \]
produces out of a quadratic form (or formation) an oriented, and in fact reducible, PD space by using Wall's realisation theorem, see \cite[Chapter 19]{Ranicki}. This uses that $n$ is at least 5, and we will not need this map in what follows. 

We recall here that an oriented geometric normal space, i.e.\ a representative of $\Omega^\mathrm{N}(X)$ consists of a space $Y$, equipped with a map $Y \to X$, an oriented stable spherical fibration $\xi$ over $Y$ and a map $\S^n \to \M(\xi)$. By the Thom isomorphism, this map gives rise to a fundamental class $[X] \in \H_n(X;\Z)$, cap product with which is \emph{not} required to be an isomorphism. 
The map from Poincar\'e bordism to normal bordism is induced by viewing a PD space as a normal space, i.e.\ by forgetting that the (co)homology of $X$ satisfies Poincar\'e duality.

We will need the following well-known properties of the Poincar\'e obstruction sequence, which justify calling the map $\vartheta$ the \emph{Poincar\'e obstruction map}. 
\begin{Lemma} 
The part 
\[ \Omega^\mathrm{P}_n(X) \lto \Omega^\mathrm{N}_n(X) \stackrel{\vartheta}{\lto} \mathrm{L}^q_{n-1}(\Z[\pi_1(X)])\]
of the above sequence is well-defined for all $n \geq 0$ and the composite is trivial for all $n\geq 0$. 
\end{Lemma}
\begin{proof}
The map 
\[ \Omega^\mathrm{N}_n(X) \stackrel{\vartheta}{\lto} \mathrm{L}^q_{n-1}(\Z[\pi_1(X)])\]
is constructed as follows, see \cite[Chapter 16]{Ranicki}. It is given by the composite
\[ \Omega^\mathrm{N}_n(X) \lto \l^n_n(\Z[\pi_1(X)]) \lto \l^q_{n-1}(\Z[\pi_1(X)])\]
where the first map is induced by taking the normal symmetric signature of a normal space, and the second map is the boundary map in the long exact sequence relating quadratic, symmetric, and normal L-groups.
We will use that the normal symmetric signature can be constructed as follows. Recall first that the normal bordism groups satisfy 
\[\Omega^\mathrm{N}_*(X) \cong \pi_*(X_+\otimes \MSG) \]
by normal transversality, see \cite{Quinn}. 
Using the \emph{normal Sullivan--Ranicki orientation} $\MSG \to \l^n(\Z)$,
one can then consider the composite
\[ X_+\otimes \MSG \lto X_+ \otimes \l^n(\Z) \stackrel{\mathcal{A}}{\lto} \l^n(\Z[\pi_1(X)])\]
where $\mathcal{A}$ denotes the assembly map for normal L-theory.
Then the above normal symmetric signature coincides with the map induced on $\pi_n$ of this composite.

Since assembly maps are natural, it follows that the diagram relating the assembly map for normal and quadratic L-theory
\[\xymatrix{X_+\otimes \l^n(\Z) \ar[r]^-{\mathcal{A}} \ar[d] & \l^n(\Z[\pi_1(X)]) \ar[d] \\ \Sigma(X_+\otimes \l^q(\Z)) \ar[r]^-{\Sigma\mathcal{A}} & \Sigma\l^q(\Z[\pi_1(X)])}\]
commutes.
The Poincar\'e obstruction map $\vartheta$ is thus given by the induced map on homotopy groups of the composite
\[ X_+ \otimes \MSG \lto X_+\otimes \l^n(\Z) \lto \Sigma( X_+\otimes  \l^q(\Z)) \stackrel{\Sigma\mathcal{A}}{\lto} \Sigma\l^q(\Z\pi_1(X))\]
as indicated in the diagram in \cite[Proposition 19.6]{Ranicki}.
It follows that one has a commutative diagram, compare to \cite[page 283]{RanickiTSO},
\[\begin{tikzcd}
	\Omega^\mathrm{P}_n(X) \ar[d] \ar[r] & \l^s_n(\Z\pi_1(X)) \ar[d]  \\ \Omega^\mathrm{N}_n(X) \ar[r] \ar[dr,"\vartheta"', bend right=20]& \l^n_n(\Z\pi_1(X)) \ar[d] \\ & \l^q_{n-1}(\Z\pi_1(X)) 
\end{tikzcd}\]
where the top horizontal map is induced by taking the symmetric signature of a Poincar\'e complex.
This shows that the Poincar\'e obstruction map always vanishes on Poincar\'e bordism since the right vertical composition is zero. 
\end{proof}

The following lemma is crucial for the following arguments. We include a proof here for completeness, the $\pi_3$-case is essentially \cite[Example 2.16 (iv)]{Ranicki}.
\begin{Lemma}\label{Lemma:effect-on-pi3}
The normal Sullivan--Ranicki orientation $\MSG \to \l^n(\Z)$ induces an isomorphism on $\pi_3$ and $\pi_4$.
\end{Lemma}
\begin{proof}
First, we record here the low dimensional homotopy groups of $\MSG$ and $\l^n(\Z)$. They are given by 
\[\pi_n(\MSG) \cong \begin{cases} \Z & \text{ for } n=0, \\ 0 & \text{ for } n=1,2 \\ \Z/2 & \text{ for } n=3 \\ \Z/8 & \text{ for } n=4. \end{cases} 
	\quad \text{ and } \quad \pi_n(\l^n(\Z)) \cong \begin{cases} \Z/8 & \text{ for } n \equiv 0 (4) \\ \Z/2 & \text{ for } n \equiv 1,3 (4) \\ 0 & \text{ for } n\equiv 2 (4) \end{cases}\]
For the homotopy of normal L-theory, see \cite[pg.\ 13]{Ranicki}. To see that the homotopy groups of $\MSG$ are as claimed, consider the cofibre $C$ of the canonical map $\MSO \to \MSG$. Using the Thom isomorphism one sees that the homology of $C$ is trivial in degrees below $3$ and isomorphic to $\Z/2$ in degree 3. Therefore, $C$ is 3-connective and $\pi_3(C) = \Z/2$. The calculation of for $n\leq 3$ then follows from the fact that $\pi_n(\MSO) = 0$ for $n =1,2,3$. For $n=4$, one can can give a direct calculation based on the fact that $\MSG$ is an Eilenberg--Mac Lane spectrum: Abbreviating $A= \pi_4(\MSG)$, using the lower homotopy calculation of $\pi_n(\MSG)$ and the fact that $\H_1(\H\Z/2;\Z) = 0$, we find an isomorphism.
\[ \H_4(\BSG;\Z) \cong \H_4(\MSG;\Z) \cong \H_4(\H\Z;\Z) \oplus A.\]
Since $\H_4(\BSG;\Z) \cong \Z/2 \oplus \Z/24$, see e.g.\ \cite[pg.\ 249]{Milgram}, and $\H_4(\H\Z;\Z) \cong \Z/2 \oplus \Z/3$, we find that $\pi_4(\MSG) = A \cong \Z/8$ as claimed. 

Invoking the commutative diagram 
\[\begin{tikzcd}
	\Omega^\mathrm{P}_4(\ast) \ar[r] \ar[d] & \l^s_4(\Z) \ar[d] \\
	\Omega^\mathrm{N}_4(\ast) \ar[r] & \l^n_4(\Z)
\end{tikzcd}\]
and the fact that $\CP^2$ has signature $1$, we deduce that the lower horizontal map is isomorphic to a surjection $\Z/8 \to \Z/8$ and therefore an isomorphism. 

To discuss the effect of the normal Sullivan--Ranicki orientation on $\pi_3$, we now recall from \cite[page 283]{RanickiTSO} that the (normal) Sullivan--Ranicki orientations fit into a commutative diagram of exact rows
\[ \begin{tikzcd}
	\Omega^\mathrm{N}_{n+1}(\ast) \ar[r] \ar[d] & \Omega^\mathrm{N,P}_{n+1}(\ast) \ar[d] \ar[r] & \Omega^{\mathrm{P}}_{n}(\ast) \ar[d] \\
	\l^n_{n+1}(\Z) \ar[r] & \l^q_{n}(\Z) \ar[r] & \l^s_{n}(\Z)
\end{tikzcd}\]
whose middle vertical map has the following property: Given a degree one normal map $f\colon M \to X$ from a closed $n$-dimensional manifold to a simply connected $n$-dimensional PD complex $X$, its mapping cylinder $(\mathrm{Cyl}(f),M \cup -X)$ is canonically an element of $\Omega_{n+1}^\mathrm{N,P}(\ast)$. On such elements, the vertical map in the above diagram simply takes the surgery obstruction $\sigma(f)$ of the degree one normal map $f$. 

We now aim to show that the left vertical map in the above diagram is surjective for $n=2$. Using that $\Omega_2^{\mathrm{P}}(\ast) = 0$, see \cref{Rmk19.9-Ranicki} below for a proof of this fact, and the fact that the map $\l^n_3(\Z) \to \l^q_2(\Z)$ is an isomorphism, it then suffices to show that the map $\Omega^\mathrm{N,P}_3(\ast) \to \l^q_2(\Z)$ is surjective. Using the above description of this map, it suffices to recall that there is a degree one normal map $T^2 \to S^2$ of Arf invariant one.
\end{proof}

\begin{Rmk}\label{Rmk19.9-Ranicki}
We now expand on Ranicki's remark \cite[Remark 19.9]{Ranicki} and will show that the following diagram of spectra
\[\begin{tikzcd}
	\Omega^\mathrm{P}(\ast) \ar[r] \ar[d] & \tau_{\geq 0}\l^s(\Z) \ar[d] \\
	\Omega^\mathrm{N}(\ast) \ar[r] & \tau_{\geq 1/2}\l^n(\Z)
\end{tikzcd}\]
is a pullback. Here, $\tau_{\geq 1/2} \l^n(\Z)$ denotes the pullback $\tau_{\geq 0} \l^n(\Z) \times_{\H\Z/8} \H\Z$, and the horizontal maps are induced by the (normal) Sullivan--Ranicki orientations. Equivalently, we will show that the canonical map 
\[\Omega^{\mathrm{P}}(\ast) \lto F = \fib\big(\Omega^\mathrm{N}(\ast) \to \Sigma \tau_{\geq 1}\l^q(\Z)\big) \]
induces an isomorphism on $\pi_n$ for $n\geq 0$. The results of Jones, see \cite{Jones, Quinn}, say that this map induces an isomorphism on $\pi_n$ for $n\geq 5$. It therefore suffices to discuss the induced map on $\pi_n$ for $n\leq 4$. We discuss the case $n=4$ first and 
consider the diagram of short exact sequences
\[\begin{tikzcd}
	 & \l^q_4(\Z) \ar[r] & \pi_4(F) \ar[r] \ar[d] & \pi_4(\MSG) \ar[d,"\cong"] \ar[r] & 0 \\
	0 \ar[r] & \l^q_4(\Z) \ar[r] \ar[u,equal] & \l^s_4(\Z) \ar[r] & \l^n_4(\Z) \ar[r] & 0
\end{tikzcd}\]
which, together with \cref{Lemma:effect-on-pi3}, shows that the map $\pi_4(F) \to \l^s_4(\Z)$ is an isomorphism. Now, in addition, the composite 
\[ \Omega^\mathrm{P}_4(\ast) \lto \pi_4(F) \lto \l^s_4(\Z)\]
is given by taking the signature of an oriented PD-complex, and is an isomorphism as one can show by a direct argument, see e.g.\ \cite[pg.\ 4]{Hillman2}. Consequently, the map $\Omega^\mathrm{P}_4(\ast) \to \pi_4(F)$ is an isomorphism. Now, \cref{Lemma:effect-on-pi3}, together with the fact that $\l^s_1(\Z) = \l^s_2(\Z) = 0$ shows that $\pi_n(F) = 0$ for $0<n<4$; see \cite[pg.\ 12 Introduction]{Ranicki} and e.g.\ \cite[Section 2.2]{CDHIII} for the general calculation of quadratic and symmetric L-groups of Dedekind rings. Therefore, it suffices to know that the same is true for $\Omega^\mathrm{P}_n(\ast)$. For $n=3$ this was recently shown by Hillman \cite[Theorem 1]{Hillman2}, and for $n=1,2$ it is classical, we record a proof here for convenience.
We recall from \cref{Rmk:2d-PD-spaces} that for $n=1,2$, every oriented n-dimensional PD-space is homotopy equivalent to a closed n-dimensional manifold. Since the mapping cylinder of a homotopy equivalence between PD spaces is a Poincar\'e cobordism, it suffices now to recall that every oriented n-dimensional closed manifold is null bordant for $n=1,2$. 
%

Finally, we note that both the composite 
\[\Omega^{\mathrm{P}}_0(\ast) \lto \pi_0(F) \lto  \Omega^\mathrm{N}_0(\ast) \]
and the latter map in the above composite are isomorphisms. This finishes the proof of the lemma.
\end{Rmk}

Continuing towards the proof of our main theorem, we will abbreviate a normal space $(X,\xi,\S^n \to \M(\xi))$ with $X$ and write $e_1(X)$ for $e_1(\xi)$.

\begin{Lemma}\label{lemma appendix}
The map 
\[ \vartheta\colon \Omega_3^\mathrm{N}(\ast) \lto \l^q_2(\Z)\cong \Z/2 \]
is given by the formula
\[ X \longmapsto \langle e_1(X),[X] \rangle.\]
\end{Lemma}
\begin{proof}
First we claim that the map $\Omega_3^\mathrm{N}(\ast) \to \Z/2$ given by
\[ X \longmapsto \langle e_1(X),[X] \rangle \]
is an isomorphism. To see this one can again consider the cofibre sequence
\[ \MSO \lto \MSG \lto C \]
and calculate the homology of $C$ to be trivial in degrees less than 3 and a $\Z/2$ in degree 3 using the Thom isomorphism for the homology of $\MSO$ and $\MSG$. The claim then follows from the fact that $e_1$ is the unique non-trivial class in $\H^3(\mathrm{BSG};\Z/2)$ which vanishes upon pulling back to $\BSO$, compare \cref{e_1}.

Since $\l^q_2(\Z) \cong \Z/2$ it will then suffice to show that the Poincar\'e obstruction map as described above is non-trivial. 
By definition, it is given by the composite
\[ \pi_3(\MSG) \lto \l^n_3(\Z) \stackrel{\cong}{\lto} \l^q_2(\Z) \]
of which the first map is an isomorphism by \cref{Lemma:effect-on-pi3} and the second map is an isomorphism as well since $\l^s_3(\Z) = \l^s_2(\Z) = 0$.
\end{proof}

As a consequence, we obtain the following corollary.
\begin{Cor}
Every oriented 3-dimensional PD space is reducible.
\end{Cor}
\begin{proof}
We have to show that if $X$ is an oriented 3-dimensional PD space, then 
\[ \langle e_1(X),[X] \rangle = 0.\]
\cref{lemma appendix} shows that the left hand side is given by viewing $X$ as an element of Poincar\'e bordism and then applying the composite
\[ \Omega_3^\mathrm{P}(\ast) \lto \Omega_3^\mathrm{N}(\ast) \lto \mathrm{L}_2^q(\Z) \cong \Z/2.\]
But the composite is trivial.
\end{proof}

To deal with the 4-dimensional case we will make use of the following lemma. We denote by $\widetilde{\Omega}_n^\mathrm{N}(X)$ the reduced normal bordism group of $X$, i.e.\ the kernel of the map $\Omega^\mathrm{N}_n(X) \to \Omega^\mathrm{N}_n(\ast)$. 
\begin{Lemma}\label{lemma}
The Poincar\'e obstruction map 
\[\widetilde{\Omega}_4^\mathrm{N}(K(\Z/2,1)) \lto \l^q_3(\Z[\Z/2]) \cong \Z/2 \]
is given by the formula
\[ \big(X \stackrel{x}{\to} K(\Z/2,1)\big) \longmapsto \langle e_1(X) \cup x, [X] \rangle.\]
\end{Lemma}
\begin{proof}
As in the proof of \cref{lemma appendix} we start by showing that the map 
\[ \big(x\colon X \to K(\Z/2,1)\big) \longmapsto \langle e_1(X) \cup x, [X] \rangle\] 
defines an isomorphism
\[ \widetilde{\Omega}_4^\mathrm{N}(K(\Z/2,1)) \cong \Z/2.\]
This follows from the Atiyah--Hirzebruch spectral sequence which shows that if $Y$ denotes the generator of $\pi_3(\MSG)$, then the element
\[ Y\times S^1 \stackrel{\mathrm{pr}}{\lto} S^1 \lto K(\Z/2,1) \]
is the non-trivial element of $\H_1(K(\Z/2,1);\pi_3(\MSG)) \cong \widetilde{\Omega}_4^{\mathrm{N}}(K(\Z/2,1))$, we use here that $\MSG$ splits as a sum of Eilenberg-Mac Lane spectra, for instance since it is an $\E_1$-$\H\Z$-algebra and therefore a module over $\H\Z$, see for instance \cite[Corollary 3.7]{HLN}. It hence suffices to see that the explicit formula of the lemma does not vanish on this specific element, where it is true by construction.
It hence again suffices to show that the reduced Poincar\'e obstruction map is non-trivial.

As explained earlier, it factors as the composite
\[ \widetilde{\Omega}_4^\mathrm{N}(K(\Z/2,1)) \lto \pi_3(K(\Z/1,1) \otimes \l^q(\Z)) \stackrel{\mathcal{A}}{\lto} \l^q_3(\Z[\Z/2]).\]
In \cite[page 109]{Ranicki} it is stated that the assembly map for quadratic L-theory factors through $\l^q_3(\Z)$ which is trivial. This is not true, though, and in fact it turns out that the assembly map
\[ \pi_3(K(\Z/2,1) \otimes \l^q(\Z)) \lto \l^q_3(\Z[\Z/2])\]
is non-trivial, see \cite[Proposition 7.3]{Weiss}. More precisely, it is shown there that 
\[ \H_1(K(\Z/2,1);\l^q_2(\Z)) \lto \pi_3(K(\Z/2,1)\otimes \l^q(\Z)) \lto \l^q_3(\Z[\Z/2]) \]
is non-zero. Since the composite
\[ \pi_3(\MSG) \lto \pi_3(\l^n(\Z)) \lto \pi_2(\l^q(\Z)) \]
is an isomorphism it follows that the reduced Poincar\'e obstruction map 
\[ \pi_4(K(\Z/2,1)\otimes \MSG) \lto \pi_3(K(\Z/2,1)\otimes \l^q(\Z)) \stackrel{\mathcal{A}}{\lto} \l^q_3(\Z[\Z/2]) \]
is non-trivial as needed.
\end{proof}

\begin{Thm}\label{4-dimensional-case}
Let $X$ be a 4-dimensional oriented PD space. Then $X$ is reducible.
\end{Thm}
\begin{proof}
We need to show that $e_1(X) = 0$. By Poincar\'e duality, this is equivalent to showing that $\langle e_1(X) \cup x,[X] \rangle$ vanishes for all $x \in \H^1(X;\Z/2)$. We view the classifying map $x\colon X \to K(\Z/2,1)$ as an element of $\Omega_4^\mathrm{P}(K(\Z/2,1))$ and consider the sequence
\[\Omega_4^\mathrm{P}(K(\Z/2,1)) \lto \Omega_4^\mathrm{N}(K(\Z/2,1)) \lto \l^q_3(\Z[\Z/2]) \cong \Z/2.\]
As explained earlier, on the one hand, the composite is trivial, but by \cref{lemma} it is also given by 
\[ \langle e_1(X) \cup x,[X] \rangle \in \Z/2.\]
Thus the claim follows.
\end{proof}

As pointed out by Hambleton \cite{Hambleton}, one can deduce the reducibility of 3-dimensional PD spaces (orientable or not) from \cref{4-dimensional-case}.

\begin{Thm}\label{3-dimensional case}
Let $X$ be a 3-dimensional PD space. Then $X$ is reducible.
\end{Thm}
\begin{proof}
Let $\hat{X} \to X$ be the principal $C_2$-bundle associated to the orientation local system of $X$. Consider the $C_2$-space $S^1$ with action given by complex conjugation. One obtains a fibre bundle 
\[ S^1 \lto \hat{X} \times_{C_2} S^1 \stackrel{p}{\lto} X \]
whose projection has a section induced by the inclusion of a fixed point for the $C_2$-action on $S^1$. From \cite[Corollary F]{Klein} it follows that $\hat{X}\times_{C_2} S^1$ is a PD-space of dimension $4$ and it is not hard to see that it is orientable since $C_2$ acts orientation reversingly on both $\hat{X}$ and $S^1$. Furthermore, \cite[Theorem I]{Klein} implies that 
\[ \SF(\hat{X}\times_{C_2} S^1) = p^*(\SF(X)) \ast S_v(p) \]
where $S_v(p)$ is the underlying sphere bundle of the vertical tangent bundle of the fibre bundle $p$. Thus $S_v(p)$ has a vector bundle reduction, and by \cref{4-dimensional-case} so does $\SF(\hat{X}\times_{C_2} S^1)$. Thus also $p^*(\SF(X))$ has a vector bundle reduction, i.e.\ the composite
\[\begin{tikzcd} \hat{X}\times_{C_2} S^1 \ar[r,"p"] & X \ar[r,"\mathrm{SF}(X)"] & \BG \ar[r] & \B(\G/\o) \end{tikzcd} \]
is trivial. Since $p$ admits a section it follows that also 
\[ \begin{tikzcd} X \ar[r,"\SF(X)"] & \BG \ar[r] & \B(\G/\o) \end{tikzcd}\]
is trivial as claimed.
\end{proof}

\section{A geometric approach to reducibility}\label{section3}

The purpose of this section is to give an alternative characterisation of reducible PD spaces using degree one normal maps, and to indicate how this approach can lead to another proof of the main result. It is in this approach where the earlier version of this paper contained a gap, as we will explain below.
We first recall the following result due to Sullivan:
\begin{Prop}\label{Sullivan}
Let $X$ be an $n$-dimensional PD space. Then $X$ is reducible if and only if there exists a degree 1 normal map $f\colon M \to X$ for a closed $n$-dimensional manifold $M$.
\end{Prop}
\begin{proof}
A degree 1 normal map $f\colon M \to X$ more precisely consists of a pullback diagram
\[\begin{tikzcd}
	\nu(M) \ar[r,"\bar{f}"] \ar[d] & E \ar[d] \\
	M \ar[r,"f"] & X
\end{tikzcd}\]
where $E$ is a stable vector bundle over $X$, $\nu(M)$ is the stable normal bundle of $M$, and $f$ has degree 1. \cref{detecting the SF} shows that the underlying spherical fibration of $E$ is the Spivak normal fibration of $X$, so $X$ is reducible. 

Conversely, if $X$ is reducible, let $E$ be a vector bundle reduction of $\SF(X)$. Then the (unstable) Pontryagin--Thom collapse map $c_X \colon S^N \to \Th(E)$ can be made transverse to the zero-section $X \to \Th(E)$, giving rise to a submanifold $M$ of $S^N$ mapping to $X$, together with an identification of the normal bundle of the embedding $M \subseteq S^N$ with the pullback of $E$. This shows that $M \to X$ refines to a degree one normal map.
\end{proof}

\begin{Rmk}
By the transversality results available for zero sections of Top and PL bundles, the above proposition (with essentially the same proof) is also correct if reducible is replaced by Top- or PL-reducible, and $M$ is a closed topological or PL manifold, respectively.
\end{Rmk}

For 4-dimensional PD spaces, one can give a sufficient criterion for the existence of a degree one normal map in terms of characteristic classes.
\begin{Prop}\label{reduction dim 4}
Let $X$ be a 4-dimensional oriented PD space. Then $X$ is reducible if and only if there exists a smooth closed oriented 4-manifold $M$ and a degree 1 map $f\colon M \to X$ such that $f^*(w_2(X)) = w_2(M)$.
\end{Prop}
\begin{proof}
By \cref{Sullivan}, if $X$ is reducible there exists a pullback diagram
\[ \begin{tikzcd} \nu(M) \ar[r,"\hat{f}"] \ar[d] & E \ar[d] \\ M \ar[r,"f"] & X \end{tikzcd}\]
where $f$ has degree 1 and $E$ is a vector bundle over $X$. Since the Stiefel--Whitney classes of a vector bundle depend only on its underlying spherical fibration, and $X$ and $M$ are orientable (in particular $w_1(X) =w_1(M) = 0$), we find that 
\[ w_2(M) = w_2(\nu_M) = f^*(w_2(E)) = f^*(w_2(\SF(X)) = f^*w_2(X).\]

To show the converse implication, by Sullivan's result it suffices to find a \emph{normal refinement} of $f$, i.e.\ an oriented vector bundle $E$ over $X$ and an isomorphism $f^*(E) \simeq \nu_M$ is stable vector bundles over $M$.

First we claim that one can find an oriented vector bundle $E$ over $X$ such that $w_2(E) = w_2(X)$. For this one considers the fibre sequence
\[ \tau_{\leq 4}(\BSO) \stackrel{w_2}{\lto} K(\Z/2,2) \xrightarrow{\beta\Sq^2} K(\Z,5)\]
and deduces that there is no obstruction of lifting an element in $\H^2(X;\Z/2)$ to a map 
\[ X \lto \tau_{\leq4}(\BSO).\] 
Using once again that $X$ is a 4-dimensional PD space it follows that there is thus no obstruction to realising any element in $\H^2(X;\Z/2)$ as $w_2$ of a stable oriented vector bundle over $X$. So pick a bundle $E$ with $w_2(E) = w_2(X)$. It has a first Pontryagin class $p_1(E) \in \H^4(X;\Z)$. There is a canonical identification $\H^4(X;\Z) \cong \Z$ because $X$ is oriented and we will suppress this identification in the notation. Since 
\[p_1(E) \equiv w_2(E)^2 = w_2(X)^2 \mod 2\]
its parity is determined by the cohomology ring of $X$ and $w_2(X)$. It follows that the parity of $f^*(p_1(E))$ and the parity of $p_1(M)$ agree:
\[ p_1(M) \equiv w_2(M)^2 = f^*(w_2(X)^2) = f^*(w_2(E)^2) \equiv f^*(p_1(E)) \mod 2.\]

Thus $f^*(p_1(E))$ differs from $p_1(M)$ by an even number. It hence suffices to argue why one can replace $E$ by an oriented vector bundle $E'$ with $w_2(E) = w_2(E')$ and $p_1(E')$ being the same number as $p_1(M)$ -- again, we identify $\H^4(X;\Z) \cong \Z \cong \H^4(M;\Z)$ using the respective orientations. Now observe that the PD space $X$ admits a degree 1 map to $S^4$ (use a general fact about CW structures on oriented connected PD spaces of dimension different from 3, see \cite[Corollary 2.3.1]{Wall} or use that $S^4 \to K(\Z,4)$ is a 4-equivalence) and recall that a generator of $\pi_4(\BSO)$ has $p_1$ equal to $\pm2$. Thus pulling back the correct generator of $\pi_4(\BSO)$ along the degree 1 map $X \to S^4$ provides an oriented vector bundle over $X$ with $w_2 = 0$ and $p_1 = 2$. Adding an appropriate multiple of this bundle to $E$ will produce the desired bundle $E'$.

We have thus explained how to construct an oriented stable vector bundle $E$ over $X$ such that $f^*(E)$ and $\nu(M)$ have the same image under the map
\[ [M,\BSO] \xrightarrow{(w_2,p_1)} \H^2(M;\Z/2)\times \H^4(M;\Z).\]

It therefore suffices to know that this map is an injective group homomorphism.
It is easy to see that this map is a group homomorphism because $\H^4(M;\Z)$ is torsion-free (recall that $p_1$ is in general only primitive up to 2-torsion). To see that it is injective it thus suffice to show that an oriented stable vector bundle $V$ over $M$ where $w_2(V) = 0 = p_1(V)$ is trivial.
For this consider the diagram
\[ \begin{tikzcd} & K(\Z,4) & \\ K(\Z,4) \ar[ur, bend left, "\cdot 2"] \ar[r] & \tau_{\leq 4}(\BSO) \ar[u,"p_1"] \ar[r,"w_2"] & K(\Z/2,2) \\ & M \ar[u,"V"] & \end{tikzcd}\]
and notice that the commutativity of the diagram is a consequence of the fact that the correct generator of $\pi_4(\BSO)$ has $p_1$ equal to $2$.
By assumption, $V$ lifts to $K(\Z,4)$, i.e.\ is given by an element $\bar{V} \in \H^4(M;\Z)$. This has the property that 
\[ 2 \cdot \bar{V} = p_1(V) = 0 \]
and thus that $\bar{V}$ is in fact zero, as $\H^4(M;\Z) \cong \Z$. Thus the claim and with it the lemma follows.
\end{proof}

We have now reduced the question of reducibility of a given PD space $X$ to the existence of specific maps from a manifold to $X$. The existence of such maps can be determined by the (twisted) Atiyah--Hirzebruch spectral sequence for spin bordism, as we explain now.

First, let us assume that $X$ is spin, i.e.\ $w_1(X) = w_2(X) = 0$. Then \cref{reduction dim 4} says that $X$ is reducible if and only if there is a degree one map $M \to X$ where $M$ is a closed spin 4-manifold. This is the case if and only if the map 
\[ \Omega^\Spin_4(X) \cong \MSpin_4(X) \lto \H_4(X;\Z) \]
induced by the canonical map $\MSpin \to \H\Z$ is surjective. In general, i.e.\ when $w_2(X)$ is not assumed to be trivial, we have to consider the \emph{twisted} spin-bordism $\Omega^\Spin_4(X;w_2)$ of $X$ with twist given by the canonical map
\[ X  \stackrel{w_2}{\lto} K(\Z/2,1) \lto \B\gl_1(\MSpin).\]
By a Pontryagin--Thom construction, elements of this twisted bordism theory can be represented by oriented manifolds $M$ equipped with a map to $X$, such that the composite
\[ M \lto X \stackrel{w_2}{\lto} K(\Z/2,2) \]
is equipped with a homotopy to $w_2(M)$.
There is a twisted Atiyah--Hirzebruch spectral sequence
\[ E^2_{p,q} = \H_p(X;\pi_q(\MSpin)) \Longrightarrow \Omega_{p+q}^{\Spin}(X;w_2).\]
If $w_2=0$ is trivial, $\Omega^\Spin_*(X;0)$ is the ordinary spin bordism group of $X$, and the twisted spectral sequence is the usual Atiyah--Hirzebruch spectral sequence for spin bordism.
In any case, the fundamental class $[X]$ of $X$ determines an element in 
\[ E^2_{4,0} = \H_4(X;\pi_0(\MSpin)).\]
Summarizing the above arguments, we arrive at the following theorem.
\begin{Thm}
Let $X$ be an oriented 4-dimensional PD space. Then $X$ is reducible if and only if $[X]$ is a permanent cycle in the twisted Atiyah--Hirzebruch spectral sequence calculating $\Omega_*^\Spin(X;w_2)$.
\end{Thm}

By \cref{4-dimensional-case} or \cite{Hambleton}, $X$ is reducible, and therefore $[X]$ is indeed a permanent cycle in this spectral sequence. Conversely, if we can show that $[X]$ is a permanent cycle by other means, reducibility of $X$ follows. Next, we analyse what can be said about this question without yet alluding to \cref{4-dimensional-case}.
For this, we recall that 
\[ \pi_n(\MSpin) \cong \begin{cases} \Z & \text{ for } n=0 \\ \Z/2 & \text{ for } n=1,2 \\ 0 & \text{ for } n=3 \end{cases}.\]
Therefore, there are two possibly non-trivial differentials emanating from $E^2_{4,0}$, the $d_2$ and the $d_3$.

\begin{Lemma}\label{Lemma-d2}
In the situation described above, we have $d_2[X] = 0$.
\end{Lemma}
\begin{proof}
The differential is given by the following composite
\[ \H_4(X;\Z) \lto \H_4(X;\Z/2) \xrightarrow{(\Sq^2_{w_2})^*} \H_2(X;\Z/2)\]
where $\Sq^2_{w_2}$ denotes the map $\H^2(X;\Z/2) \to \H^4(X;\Z/2)$ given by
\[ x \longmapsto \Sq^2(x) + w_2 \cup x,\]
see \cite[Proposition 1]{Teichner} for a similar statement.
The vanishing of the $d_2$ thus amounts to the statement that for every element $x \in \H^2(X;\Z/2)$ one has 
\[ \langle \Sq^2(x),[X] \rangle = \langle w_2\cup x, [X] \rangle \]
which is a consequence of the Wu--formula and the assumption that $X$ is oriented, so that $w_2(X) = v_2(X)$.
\end{proof}

\begin{Rmk}\label{rem:target-of-d3-trivial}
In order to show that $[X]$ is a permanent cycle in the above spectral sequence, the task then lies in showing the vanishing of the $d_3$-differential
\[ \H_4(X;\Z) \lto \H_1(X;\Z/2)/d_2(\H_3(X;\Z/2)).\]
As in the proof of \cref{Lemma-d2}, the $d_2$-differential appearing in the codomain is determined by the formula
\[ \langle d_2(\alpha),x \rangle = \langle w_2(X)\cup x, \alpha \rangle \]
for all $x\in \H^1(X;\Z/2)$ and $\alpha \in \H_3(X;\Z/2)$. In particular, the $d_3$ vanishes if its target vanishes, which is the case if the map $w_2(X) \cup - \colon \H^1(X;\Z/2) \to \H^3(X;\Z/2)$ is injective (and hence an isomorphism by Poincar\'e duality, but we shall ignore this). Using the Wu formula, this condition is equivalent to the condition that the bilinear form on $\H^1(X;\Z/2)$ given by
\[ (x,y) \mapsto \langle x^2 \cup y^2 ,[X] \rangle \]
is non-degenerate, i.e.\ induces an injection $\H^1(X;\Z/2) \to \H_1(X;\Z/2)$. We note that this condition is not implied by the non-triviality of $w_2$, as the example of $\mathbb{CP}^2 \# (S^1 \times S^3)$ shows.
\end{Rmk}

As another extreme case, one can consider the case where $w_2 = 0$. Also in this case, one can give a direct argument for the vanishing of the needed $d_3$-differential. We wish to thank Achim Krause for a nice Skype session about this.
\begin{Prop}
Let $X$ be a spin 4-dimensional PD space. Then $d_3[X]=0$, and hence $[X]$ is a permanent cycle in the twisted Atiyah--Hirzebruch spectral sequence calculating $\Omega^\Spin_*(X;w_2)$.
\end{Prop}
\begin{proof}
As just explained, we need to show the vanishing of the differential 
\[ d_3 \colon \H_4(X;\Z) \lto \H_1(X;\Z/2),\]
in the usual Atiyah--Hirzebruch spectral sequence converging to $\Omega^\Spin_*(X)$, which is implied by the statement that 
\[ \langle d_3(\alpha),x \rangle = 0 \]
for all $\alpha \in \H_4(X;\Z)$ and $x \in \H^1(X;\Z/2)$. We represent $x$ by a map $\vartheta_x \colon X \to K(\Z/2,1)$ and consider the morphism of spectral sequences induced by the map $\vartheta_x$. This induces a commutative square
\[ \begin{tikzcd} \H_4(X;\Z) \ar[r,"d_3"] \ar[d,"(\vartheta_x)_*"'] & \H_1(X;\Z/2) \ar[d,"(\vartheta_x)_*"] \\ 0 =\H_4(K(\Z/2,1);\Z) \ar[r,"d_3"] & \H_1(K(\Z/2,1);\Z/2) \end{tikzcd}\]
where the vertical maps are induced by $\vartheta_x$. This simply follows because all $d_2$-differentials involving groups that occur in this diagram vanish since $\Sq^2$ vanishes on classes of degree 1 and the already established fact that $[X]$ is in the kernel of the $d_2$. Since the lower left corner of this diagram is trivial it follows that 
\[ (\vartheta_x)_*(d_3[X]) = 0 \]
for all $x \in \H^1(X;\Z/2)$. This implies that $d_3[X]$ is zero as needed.
\end{proof}

\begin{Rmk}
It would be nice to find an independent proof of the vanishing of the $d_3$-differential on $[X]$ for a general oriented 4-dimensional PD space $X$.
\end{Rmk}

\begin{Warning}\label{Warning:error}
In an earlier version of this paper, I intended to show the vanishing of the above $d_3$-differential by comparing $\MSpin$ to $\MSGpin$, where the latter is the Thom spectrum of the fibre of the map 
\[\B\G \stackrel{(w_1,w_2)}{\lto} K(\Z/2,1)\times K(\Z/2,2).\]
Indeed, one can show that $[X]$ is a permanent cycle in the twisted spectral sequence converging to $\Omega^{\mathrm{SGpin}}_*(X;w_2)$, so the result would follow if the canonical map $\MSpin \to \MSGpin$ induced an isomorphism on $\pi_i$ for $i \leq 2$. 
This was claimed as Lemma 3.2 in an earlier version, but is wrong, as was pointed out to us by Diarmuid Crowley, in fact one has $\pi_2(\MSGpin)=0$, whereas $\pi_2(\MSpin) = \Z/2$. 
I wish to thank Diarmuid for explaining to me the following proposition, of which the just described vanishing is the special case $n=2$. To fix notation, we let $\B\G\langle n+1 \rangle = \mathrm{fib}(\B\G \to \tau_{\leq n}\B\G)$ and let $\M\G\langle n+1\rangle$ be the Thom spectrum of the tautological stable spherical fibration on $\B\G\langle n+1\rangle$. For instance, we have $\M\G\langle 1 \rangle = \M\G$, $\M\G\langle 2 \rangle = \mathrm{MSG}$, and $\M\G\langle 3 \rangle = \MSGpin$.
\begin{Prop}
For every $n\geq 1$, we have $\pi_n(\M\G\langle n+1\rangle) = 0$.
\end{Prop}
\begin{proof}
Since $n\geq 1$, we find that the tautological spherical fibration on $\BG\langle n+1\rangle$ is orientable, so the Thom isomorphism shows that 
\[ \H_i(\M\G\langle n+1 \rangle;\Z) \cong \H_i(\B\G\langle n+1 \rangle ;\Z) \]
vanishes for $0 <i \leq n$. Therefore, the Hurewicz theorem implies that the unit map $\S \to \M\G\langle n+1 \rangle$ of the ring spectrum $\M\G\langle n+1\rangle$ induces an isomorphism on $\pi_i$ for $i <n$ and surjection on $\pi_n$. To show the proposition, it therefore suffices to show that the map $\pi_n(\S) \to \pi_n(\M\G\langle n+1 \rangle)$ is also the zero map.
Now, via the isomorphism $\pi_n(\S) \cong \pi_{n+1}(\B\G)$, an element $x$ of $\pi_n(\S)$ determines an associated stable spherical fibration $\xi_x$ over $S^{n+1}$. We shall make use of the following well-known lemma, and will give a proof of it in modern language in \cref{proof-of-lemma}
\begin{Lemma}\label{Lemma:cell-structure-on-Thom-spectrum}
There is a canonical cofibre sequence $\S^n \stackrel{\cdot x}{\lto} \S \lto \M(\xi_x)$ of spectra.
\end{Lemma}
Now we notice that $S^{n+1}$ is $n$-connected, so that the classifying map $\xi \colon S^{n+1} \to \B\G$ lifts to a map $\xi'\colon S^{n+1} \to \B\G\langle n+1 \rangle$, and therefore that the unit map $\S \to \M\G\langle n+1 \rangle$ factors as the composite
\[ \S \lto \M(\xi_x) \lto \M\G\langle n+1 \rangle.\]
Therefore, the above lemma implies that the map $\pi_n(\S) \to \pi_n(\M\G\langle n+1 \rangle)$ is the zero map as needed.
\end{proof}
\end{Warning}

\begin{Rmk}
One can also give a proof of \cref{3-dimensional case} using the geometric approach presented in this section, as was explained to us by Wolfgang L\"uck. More precisely, in analogy to \cref{reduction dim 4} one first shows that it suffices to find a degree $1$ map $f\colon M \to X$ with $f^*(w_1(X)) = w_1(M)$. The twisted Atiyah--Hirzebruch spectral sequence converging to $\Omega_3^{\SO}(X;w_1)$ shows such a map to exist: There are no possible differentials on the fundamental class because $\pi_i(\MSO) = 0$ for $i \in \{1,2,3\}$, see \cite[Section 3]{Cannon-conjecture}.
\end{Rmk}

\section{Further remarks}\label{section4}
\subsection{Non-orientable complexes of dimension 4}
First, we  want to mention that a construction of Hambleton and Milgram gives a non-orientable 4-dimensional PD space whose Spivak normal fibration is not reducible. Indeed, in \cite[section 4]{HM} of loc.\ cit.\ the Spivak normal bundles of spaces called $X^4$ and $X^6$ are discussed. We shall only be interested in $X^4$. 
It is a space obtained from $\RP^2 \vee S^2$ by attaching one 3-cell and one 4-cell, see \cite[line 7 page 1325]{HM}. It follows that, pinching the 2-skeleton of $X$, one obtains a cofibre sequence
\[S^3 \stackrel{\cdot 2}{\lto} S^3 \lto X/X^{(2)} \]
because $X$ is non-orientable and thus must have top cohomology (integrally) isomorphic to $\Z/2$.

We note that the above cell structure implies that $\pi_1(X) \cong \Z/2$, so there is a unique non-trivial real line bundle $\eta_1$ over $X$. It is then shown, see \cite[Corollary 4.3]{HM}, that the Spivak normal fibration of $X$ is a sum of two spherical fibrations, namely $\eta_1 \oplus (\kappa)$ (in their notation) where $(\kappa)$ is a spherical fibration making the diagram
\[ \begin{tikzcd} S^3 \ar[r] \ar[dr, bend right, "e_1"] & X/X^{(2)} \ar[d] & X \ar[l] \ar[dl, bend left, "(\kappa)"'] \\ & \BG \end{tikzcd}\]
commute. Here $e_1$ again denotes the exotic class as considered in \cref{section1}. It follows that $e_1(\kappa) \neq 0$ as both maps 
\[ S^3 \lto X/X^{(2)} \longleftarrow X \]
induce an isomorphism in third mod 2 cohomology. We conclude that $e_1(\SF(X)) \neq 0$: Suppose to the contrary that $e_1(\SF(X)) = 0$. Then there exists a vector bundle reduction of $\eta_1 \oplus (\kappa)$. Since $\eta_1$ is realised by a vector bundle and $\eta_1 \oplus \eta_1$ is trivial (recall that for $\gamma$ the universal line bundle over $\RP^\infty$, one has that $\gamma\oplus \gamma$ is trivial) it would follow that $(\kappa)$ also admits a vector bundle reduction, contradicting $e_1(\kappa) \neq 0$.

We also want to mention that since $\pi_1(X) \cong \Z/2$, the orientation double cover is simply connected and hence has reducible Spivak fibration by \cref{reducibility for pi1 trivial 4 PD spaces} without alluding to the main theorem of this note or \cite{Hambleton}.

\subsection{Relation to surgery theory}\label{subsection-surgery}
As indicated, the question whether or not a PD complex $X$ is reducible is the first obstruction in surgery theory. Indeed, by Sullivan's result \cref{Sullivan}, a reducible PD complex $X$ admits a degree one normal map $M \to X$. One can then try to improve this map, by surgeries on $M$, to become a homotopy equivalence. For instance, if $X$ has odd dimension greater or equal to $5$ and is simply connected, the surgery obstruction groups vanish, and hence $X$ is homotopy equivalent to a closed smooth manifold. 

Contrary to this, Wall \cite{Wall} constructed examples of oriented 4-dimensional PD complexes $X$ with cyclic fundamental group of order $p$ which satisfy 
\[ \sigma(\tilde{X}) \neq p\cdot \sigma(X), \]
where $\sigma(-)$ denotes the signature and $\tilde{X}$ the universal cover of $X$. By \cite{Hambleton} or the main result of this paper we know that $X$ is reducible. But $X$ is not homotopy equivalent to a closed manifold as the signature is multiplicative for finite covers of closed manifolds, essentially since the signature can be expressed in terms of rational Pontryagin classes by Hirzebruch's signature theorem.

\subsection{Higher dimensional complexes}
We now want to show how to use the geometric method of \cref{section3} to show that there exists a 5-dimensional PD complex $X$ which is not reducible. The example is of course well-known and several proofs are possible. 
The complex $X$ is given as follows.
\[ X = (S^2 \vee S^3)\cup_\Theta D^5 \]
where $\Theta = [\iota_2,\iota_3]+\eta^2 \in \pi_4(S^2\vee S^3)$. To be precise, $\eta^2$ refers to the composite 
\[S^4 \stackrel{\eta^2}{\lto} S^2 \lto S^2\vee S^3.\]
It is not hard to check that $X$ satisfies Poincar\'e duality.

Notice that $X$ is stably equivalent to $C(\eta^2) \vee  S^3$ since $\Theta = \eta^2$ stably, and that thus that 
\[\Sq^2\colon \H^3(X;\Z/2) \lto \H^5(X;\Z/2)\] 
is trivial. Since $w_1(X) =0$, the Wu--formula implies that $w_2(X) = v_2(X) = 0$. If $X$ is reducible, then there exists a 5-manifold $M$ mapping by a degree 1 normal map to $X$, as explained in the proof of \cref{reduction dim 4}. In particular, there exists a degree 1 map $f \colon M \to X$ where $M$ is a spin 5-manifold. We will show that no such map exists, in other words
we aim at showing that the canonical map
\[ \Omega_5^{\Spin}(X) \lto \H_5(X;\Z) \]
does not hit the fundamental class of $X$. Looking at the AHSS, we can neglect the $S^3$-summand of $\Sigma^\infty_+X$ by naturality of the differentials. It then suffices to know that the $d_3$-differential
\[ d_3 \colon \H_5(C(\eta^2);\Z) \lto \H_2(C(\eta^2);\Z/2) \]
is non-trivial, as then the generator of $\H_5(X;\Z)$ is not a permanent cycle and hence not in the image of the above canonical map. The non-triviality of the $d_3$-differential follows from the fact that $\eta^2$ is stably essential and detected precisely by the secondary operation given by this $d_3$-differential. Notice that we use here (to some extent) that the unit $\S \to \MSpin$ of the ring spectrum $\MSpin$ is a 3-equivalence, so we can import differentials from the AHSS converging to stable homotopy groups of $C(\eta^2)$.

\begin{Rmk}
Other proofs of the non-reducibility of $X$ that we are aware of explicitly use that the Spivak normal fibration over $X$ is non-trivial. Of course, we use this implicitly here as well: the attaching map of the top cell of an oriented PD space is stably null-homotopic if and only if the Spivak normal fibration is trivial, see for instance \cite[Lemma 3.10]{KLPT} and thus all proofs crucially use the fact that $\eta^2$ is stably essential. 
\end{Rmk}

Let us quickly comment on the obstruction theoretic point of view for this example. As mentioned earlier, stably we have that $X$ is equivalent to $S^3 \vee C(\eta^2)$. Taking the defining cofibre sequence and mapping that to $\BO$ and $\BG$, one shows that the primary obstruction to finding a vector bundle lift (i.e.\ the element $e_1(X) \in \H^3(X;\Z/2)$) is in fact the only obstruction: The diagram implies that the canonical map 
\[ [C(\eta^2),\BO] \lto [C(\eta^2),\BG] \]
is a bijection. It follows that the Spivak fibration of $X$, viewed as an element of 
\[ [X,\BG] \cong \pi_3(\BG) \oplus [C(\eta^2),\BG] \cong \pi_3(\BG) \oplus \pi_2(\BG) \]
has a non-trivial component in $\pi_3(\BG)$. Thus $e_1(X)$ is non-trivial, but as explained earlier, $w_2$ of the Spivak fibration vanishes. Thus all Stiefel--Whitney classes (and Wu classes) of $X$ vanish, but $e_1$ does not.

\appendix

\section{A parametrised homotopy theory view on PD spaces}\label{appendix}

The purpose of this appendix is to give a short treatment of PD complexes in the language of $\infty$-categories.
Most of what we write is already contained in \cite{Klein} and we claim no originality for the material. 
\newline

An important $\infty$-category is the $\infty$-category associated to topological spaces, i.e.\ the localisation of the usual category of topological spaces at the weak homotopy equivalences. This $\infty$-category is traditionally called the $\infty$-category of spaces. It is, however, confusing to do so: Its objects are rather homotopy types than actual topological spaces. Therefore, we follow the recent terminology used in \cite{CS} and call the objects of the $\infty$-category associated to topological spaces \emph{anima} and the $\infty$-category itself the $\infty$-category of \emph{animae}, written $\Spc$.

Furthermore, we denote by $\Sp$ the $\infty$-category of spectra. We recall that $\Sp$ admits a symmetric monoidal structure, the smash product of spectra  which we however write as $\otimes$, uniquely determined by the property that it preserves colimits in each variable and that the suspension spectrum functor $\Sigma^\infty_+ \colon \Spc \to \Sp$ refines to a symmetric monoidal functor (with respect to the cartesian monoidal structure on $\Spc$).
An object in $\Spc$ or $\Sp$ is called a \emph{finite} anima, respectively spectrum, if it is contained in the smallest subcategory of $\Spc$, respectively of $\Sp$, closed under finite colimits and containing the point, respectively the sphere spectrum, and it is called \emph{compact} if it is a retract of a finite object. This agrees with the general notion of compact objects in an $\infty$-category $\C$, which are those objects $X$ for which the functor $\Map_\C(X,-)\colon \C \to \Spc$ preserves filtered colimits. The full subcategory of compact objects of an $\infty$-category $\C$ will be denoted $\C^\omega$.
A finite CW complex $X$ determines a finite anima, and the suspension spectrum functor $\Sigma^\infty_+ \colon \Spc \to \Sp$, by virtue of preserving colimits and sending the point to the sphere spectrum, sends finite anima to finite spectra, and compact anima to compact spectra. We note here that whether or not a compact anima is finite is determined by its finiteness obstruction in the sense of Wall, which vanishes for simply connected anima. From this (or otherwise) one can show that a compact spectrum is finite, but we shall not make use of this fact. With these definitions, every compact object in $\Sp$ is dualisable, and vice versa it turns out that dualisable objects are compact.

\subsection*{Recollection on parametrised spectra}
For an anima $X$, we will consider the stable $\infty$-category $\Fun(X,\Sp)$, which we will refer to as the $\infty$-category of $X$-parametrised spectra. We will denote the mapping spectrum in this category by $\map_X(-,-)$.
The category $\Fun(X,\Sp)$ is endowed with the pointwise symmetric monoidal structure denoted by $\otimes_X$ or simply $\otimes$ if $X$ is understood.
Given a map $f\colon X \to Y$, the restriction functor 
\[ f^*\colon \Fun(Y,\Sp) \lto \Fun(X,\Sp) \]
is canonically symmetric monoidal, and preserves limits and colimits. Therefore it has a right adjoint $f_*$ and a left adjoint $f_!$, given by right and left Kan extension, respectively. In the extreme case $r \colon X \to \ast$ where $Y$ is a point, the functors $r_*$ and $r_!$ simply take a limit and a colimit over $X$, respectively. We define a candidate of an internal mapping object as follows: Given $\F,\g \colon X \to \Sp$, consider the composite
\[ X \lto X\times X \simeq X^\op \times X \stackrel{\F^\op,\g}{\lto} \Sp^\op \times \Sp \stackrel{\map}{\lto} \Sp \]
which is the functor taking a point $x$ to the spectrum of maps $\map_\Sp(\F_x,\g_x)$ with functoriality given by conjugation with morphisms in $X$. We denote this composite by $\hom_X(\F,\g)$. Straight from the definitions, we obtain a canonical equivalence
\[ \hom_X(\F,\hom_X(\g,\h)) \simeq \hom_X(\F\otimes_X \g,\h).\]
In particular, by forming limits over $X$, we obtain an equivalence
\[ r_*\hom_X(\F,\hom_X(\g,\h)) \simeq r_*\hom_X(\F\otimes_X \g,\h).\]
Now we claim that $r_*\hom_X \simeq \map_X$, showing that $\hom_X(\F, -)$ is right adjoint to $\F\otimes_X -$: To show the claim, we recall that the spectrum of maps in a functor category to a stable category can be calculated by an end-formula, i.e.\ we have
\[ \map_X(\F,\g) = \lim\limits_{\Tw(X)} \map_\Sp(\F(x),\g(y))\]
where $\Tw(X)$ is the twisted arrow category of $X$, \cite{GHN}. Since $X$ is a groupoid, we find that the map $\Tw(X) \to X^\op\times X$ is equivalent to the diagonal $X \lto X\times X$. Therefore, the right hand side of the above equivalence is precisely the limit over $X$ of the functor $\hom_X(\F,\g)$ as claimed.

Again straight from the definitions, we see that for a map $f\colon X \to Y$ the functor $f^*$ is compatible with the internal mapping object in the sense that 
\[ f^*\hom_Y(\F,\g) \simeq \hom_X(f^*\F,f^*\g).\]
Symmetric monoidal functors between closed symmetric monoidal categories which are also compatible with the internal mapping objects as just described are also called closed symmetric monoidal functors. It is a completely formal consequence that the left adjoint $f_!$ of $f^*$ then satisfies the projection formula
\[ f_!(\F) \otimes_Y \g \simeq f_!(\F \otimes_X f^*\g).\]

\begin{Lemma}\label{Lemma:compact-index-anima}
Let $X$ be a compact anima. Then the limit functor $\lim_X\colon \Fun(X,\Sp) \to \Sp$ commutes with colimits. Likewise, the colimit functor $\colim_X \colon \Fun(X,\Sp) \to \Sp$ commutes with limits.
\end{Lemma}
\begin{proof}
We prove the first part, the second is analogous. Since $X$ is compact, there exists a finite space $Y$ and maps 
\[ X \stackrel{i}{\lto} Y \stackrel{p}{\lto} X \]
whose composite $pi$ is equivalent to the identity of $X$. Now let $I \to \Fun(X,\Sp)$ be a functor sending $i$ to $\alpha_i$. We note that $i^*$ and $p^*$ commute with colimits, and consider the commutative diagram
\[ \begin{tikzcd}
	\colim\limits_{i\in I} \lim\limits_X i^*p^*\alpha_i \ar[r] \ar[d] & \colim\limits_{i \in I} \lim\limits_Y p^*(\alpha_i) \ar[r] \ar[d] & \colim\limits_{i\in I} \lim\limits_X \alpha_i \ar[d] \\
	\lim\limits_X \colim\limits_{i\in I} i^*p^*(\alpha_i) \ar[r] & \lim\limits_Y \colim\limits_{i\in I} p^*(\alpha_i) \ar[r] & \lim\limits_X \colim\limits_{i \in I} \alpha_i
\end{tikzcd}\]
showing that the right vertical map is a retract of the middle vertical map. Therefore, it suffices to show the lemma for finite $Y$, in which case it follows from the fact that in stable categories, finite limits commute with colimits.
\end{proof}

\begin{Lemma}\label{Lemma:pointwise-compact-is-compact}
Let $X$ be a compact anima. Then there is a canonical inclusion 
\[ \Fun(X,\Sp^\omega) \subseteq \Fun(X,\Sp)^\omega.\]
\end{Lemma}
\begin{proof}
Let $\F \colon X \to \Sp^\omega$ and let $\g \colon \mathscr{I} \to \Fun(X,\Sp)$ be a filtered diagram of objects of $\Fun(X,\Sp)$. We need to show that the canonical map  
\[ \colim_i \map_X(\F,\g_i) \lto \map_X(\F,\colim \g_i) \]
is an equivalence, where $\map_X(-,-)$ denotes the mapping spectrum in the stable $\infty$-category $\Fun(X,\Sp)$. Again, the mapping spectrum can be calculated by an end-formula, and using once more the equivalence of the map $\Tw(X) \to X^\op \times X$ with the diagonal $X\to X \times X$, we have
\[ \map_X(\F,\g) = \lim\limits_{X} \map_\Sp(\F(x),\g(x))\]
Since $X$ is compact, the limit over $X$ commutes with colimits, by \cref{Lemma:compact-index-anima}. In particular we obtain the following chain of equivalences
\begin{align*}
	 \colim_i \map_X(\F,\g_i) & \simeq \colim_i \lim\limits_{\Tw(X)} \map(\F(x),\g_i(y)) \\ 
	 	& \stackrel{\simeq}{\lto} \lim\limits_{\Tw(X)} \colim_i \map(\F(x),\g_i(y)) \\
		& \stackrel{\simeq}{\lto} \lim\limits_{\Tw(X)} \map(\F(x),\colim \g_i(y)) \\
		& \simeq \map_X(\F,\colim \g_i) 
\end{align*}
where we have used that the values $\F(x)$ are compact spectra in the second to last equivalence.
\end{proof}

\begin{Rmk}
For sake of completeness, we record here that if $G$ is a finite group, then there is the reverse inclusion
\[ \Fun(BG,\Sp)^\omega \subseteq \Fun(BG,\Sp^\omega).\]
Here, $BG$ refers to the category with one object and $G$ as endomorphisms. The same holds true for $\Sp$ replaced by the derived category $\mathscr{D}(R)$ of a ring $R$, and in this case, the Verdier quotient (or its idempotent completion, depending on conventions) is equivalent to the stable module category $\mathrm{stmod}^R_G$ of $G$ with respect to $R$ which is an interesting object in modular representation theory.
\end{Rmk}

\subsection*{The dualising spectrum}
As indicated earlier, we will view $\Fun(X,\Sp)$ as a symmetric monoidal $\infty$-category with the pointwise tensor product of spectra and write $\S_X$ for the unit.
The symmetric monoidal functor
\[ r^* \colon \Sp \lto \Fun(X,\Sp) \]
has a left adjoint, $r_!$, and a right adjoint $r_*$, given by forming the colimit and limit of a functor $\F \colon X \to \Sp$, respectively. \cref{Lemma:compact-index-anima} says that if $X$ is compact then (co)limits over $X$ behave like \emph{finite} (co)limits, so that $r_!$ preserves limits and likewise that $r_*$ preserves colimits. 
It follows from general Morita theory that the functor 
\begin{equation}\label{Morita-theory}
\begin{tikzcd}[row sep=tiny] 
	\Fun(X,\Sp) \ar[r] & \Fun^\l(\Fun(X,\Sp),\Sp) \\
	\F \ar[r, mapsto] & r_!(\F \otimes -)
\end{tikzcd}
\end{equation}
is an equivalence of categories, where the superscript L refers to the full subcategory of the functor category consisting of the colimit preserving (or equivalently left adjoint) functors. Hence, for a compact anima $X$, the functor $r_*$ is equivalent to $r_!(D_X\otimes -)$ for some essentially unique $X$-parametrised spectrum $D_X$. In addition, the unit of the adjunction $(r^*,r_*)$ provides a canonical map 
\[ \S \lto r_*r^*(\S) \simeq r_!(D_X\otimes \S_X)) = r_!(D_X)\]
called the Pontryagin--Thom collapse map and denoted by $c_X$, where $\S_X = r^*(\S)$ is the tensor unit of $\Fun(X,\Sp)$.
\begin{Def}
Let $X$ be a compact anima. The object $D_X$ is called the dualising spectrum of $X$ and the map $c_X \colon \S \to r_!(D_X)$ is called the Pontryagin--Thom collapse map.
\end{Def}

Any object $\F \in \Fun(X,\Sp)$ equipped with a map $c\colon \S \to r_!(\F)$
determines a canonical transformation $r_*(-) \Longrightarrow r_!(\F \otimes -)$, by means of the following composite of natural transformations
\begin{align*}
r_*(-) & \simeq \map_X(\S_X,-) \stackrel{\F\otimes-}{\lto} \map_X(\F,\F\otimes -) \\ 
	&\stackrel{r_!}{\lto}  \map(r_!(\F), r_!(\F\otimes -)) \stackrel{c^*}{\lto} \map(\S,r_!(\F\otimes -)) \simeq r_!(\F\otimes-),
\end{align*}
and we say that $c$ exhibits $\F$ as the dualising spectrum $D_X$ of $X$ if this transformation is an equivalence.

Now we expand slightly on the above situation. Namely, we consider the following construction.
\[ \begin{tikzcd}[row sep=tiny]
	\Fun(X,\Sp)^\omega \ar[r] & \Fun^\l(\Fun(X,\Sp),\Sp) \stackrel{\eqref{Morita-theory}}{\simeq} \Fun(X,\Sp)\\
	\cE \ar[r, mapsto] & \map_X(\cE,-) \simeq r_!(T(\cE) \otimes -)
\end{tikzcd}\]
This gives rise to a functor $T\colon (\Fun(X,\Sp)^\omega)^\op \to \Fun(X,\Sp)$ given by sending $\cE$ to $T(\cE)$
and which preserves the objects $i_!(\S)$. Indeed, we have
\[ r_!(i_!(\S) \otimes \F) \simeq r_!i_!(\S \otimes i^*(\F)) \simeq i^*(\F) \simeq \map_X(i_!(\S),\F) \]
by the projection formula. Since the objects $i_!(\S)$ generate $\Fun(X,\Sp)^\omega$, $T$ restricts to a functor
\[ D_{CW} \colon (\Fun(X,\Sp)^\omega)^\op \lto \Fun(X,\Sp)^\omega \]
which comes with natural equivalences
\[ \map_X(\cE,D_{CW}(\F)) \simeq r_!(D_{CW}(\cE) \otimes D_{CW}(\F)) \simeq \map_X(\F,D_{CW}(E)) \]
for $\cE,\F \in \Fun(X,\Sp)^\omega$. From this, one deduces that $D_{CW}$ is adjoint to the opposite of $D_{CW}$. Furthermore, one can see that $D_{CW}$ is an equivalence of categories, and hence a duality, in particular that $D_{CW}^2\simeq \id$. This duality is called Costenoble--Waner duality (hence the notation). We summarise the situation in the following lemma.

\begin{Lemma}
Let $X$ be a compact anima, and let $\cE \in \Fun(X,\Sp)^\omega$ and $\F \in \Fun(X,\Sp)$. Then we have a natural equivalence of functors $\map_X(\cE,\F) \simeq r_!(D_{CW}(\cE)\otimes \F)$. In particular, we have $D_X \simeq D_{CW}(\S_X)$ and a canonical equivalence of functors $\map_X(D_X,-) \simeq r_!(-)\colon \Fun(X,\Sp) \to \Sp$.
\end{Lemma}
\begin{proof}
The first equivalence is valid by definition of Costenoble--Waner duality. Then we note that $\map_X(\S_X,-)$ is equivalent to $r_*(-)$ by adjunction, so $D_X \simeq D_{CW}(\S_X)$ by definition of $D_X$. We then deduce that 
\[ \map_X(D_X,-) \simeq r_!(D_{CW}(D_X)\otimes -) \simeq r_!(D^2_{CW}(\S_X) \otimes -) \simeq r_!(-)\]
since $D^2_{CW}$ is equivalent to the identity and $\S_X$ is the tensor unit of $\Fun(X,\Sp)$.
\end{proof}

As a consequence, we obtain a form of Atiyah-duality, namely an identification of the Spanier--Whitehead dual of the suspension spectrum of a compact anima in terms of its dualising spectrum.
\begin{Lemma}\label{Lemma:general-Atiyah-duality}
Let $X$ be a compact anima. Then there is a canonical equivalence
\[ D(X_+) \simeq r_!(D_X).\]
\end{Lemma}
\begin{proof}
We consider the mapping spectrum $\map_X(\S_X,\S_X)$. Using $\S_X =r^*(\S)$, the adjunction $(r^*,r_*)$ and the defining property of the dualising spectrum, we obtain an equivalence
\[ \map_X(r^*(\S),r^*(\S)) \simeq \map_\Sp(\S,r_*(r^*\S)) = r_!(D_X).\]
On the other hand, using the $(r_!,r^*)$ adjunction, we obtain an equivalence
\[ \map_X(r^*(\S),r^*(\S)) \simeq \map_X(r_!r^*(\S),\S).\]
The lemma then follows from the equivalence $r_!r^*(\S) \simeq \Sigma^\infty_+ X$, whose verification we leave to the reader, and the fact that $DY= \map(Y,\S)$ for spectra $Y$. 
\end{proof}

We recall now that $\Pic(\S)$ denotes the full subgroupoid of $\Sp^\simeq$ on $\otimes$-invertible spectra, i.e.\ spectra equivalent to $\S^n$ for some $n \in \Z$.
We recall that there are inclusions 
\[ \Fun(X,\Pic(\S)) \subseteq \Fun(X,\Sp^\omega) \subseteq \Fun(X,\Sp)^\omega.\]
An object of the full subcategory $\Fun(X,\Pic(\S))$ of $\Fun(X,\Sp)$ is called a spherical fibration.
The dimension of a spherical fibration $\F$ is the (locally constant) function $X \to \Z$ which assigns to a point $x$ the dimension of the sphere $\F_x$.
\begin{Def}\label{Def:PD-space}
A compact anima $X$ is a called a PD complex if its dualising spectrum is a spherical fibration. In this case, its dualising spectrum is also called the Spivak normal fibration of $X$.
The dimension of $X$ is defined to be the negative of the dimension of its Spivak normal fibration.
\end{Def}

\begin{Rmk}
For a spherical fibration $\xi \in \Fun(X,\Pic(\S))$, the spectrum $r_!(\xi)$ is in fact a well-known object: It is the Thom spectrum $\M\xi$ of $\xi$. Therefore, \cref{Lemma:general-Atiyah-duality} says that the Spanier--Whitehead dual of a Poincar\'e duality complex is equivalent to the Thom spectrum of its Spivak normal fibration.
\end{Rmk}

\begin{Rmk}
We recall from \cref{Lemma:pointwise-compact-is-compact} the inclusion
\[ \Fun(X,\Sp^\omega) \subseteq \Fun(X,\Sp)^\omega \]
for a compact anima $X$. The object $\S_X$ is an object of $\Fun(X,\Sp^\omega)$, so that $D_X = D_{CW}(\S_X)$ is always an object of $\Fun(X,\Sp)^\omega$. One may wonder under what conditions $D_X$ is an object of $\Fun(X,\Sp^\omega)$. This is in fact the case if and only if $X$ is a PD complex in the sense of \cref{Def:PD-space}. In other words, $D_X$ is pointwise compact if and only if it is pointwise invertible, as we will show now. To see the non-trivial statement, let us assume that $D_X$ is pointwise compact, and hence pointwise dualisable. Let $D_X^\vee$ denote the pointwise dual of $D_X$. This is then the dual of $D_X$ in the pointwise symmetric monoidal structure on $\Fun(X,\Sp)$. In particular, one has equivalences
\[ r_!(\F) \simeq \map_X(D_X, \F) \simeq \map_X(\S_X,D_X^\vee \otimes \F) \simeq r_*(D_X^\vee\otimes \F) \simeq r_!(D_X\otimes D_X^\vee \otimes \F),\]
natural in $\F$. Therefore, $D_X\otimes D_X^\vee$ is equivalent to $\S_X$, and hence $D_X$ is pointwise invertible as claimed.
\end{Rmk}

Next, we show that a Poincar\'e duality space as defined in the beginning of this docusment is indeed a PD complex in the sense of \cref{Def:PD-space}, see \cref{Prop:same-definitions}.
Before doing so, we state the following recognition principle for the Spivak normal fibration and its consequences for closed manifolds.
\begin{Prop}\label{Prop:recognition-principle-for-Spivak-fibration}
Let $X$ be a finite anima equipped with an orientation local system $\Loc$ and a class $[X] \in \H_n(X;\Loc)$ such that for any other local system $\m$ of abelian groups on $X$, the map 
\[ - \cap [X] \colon \H^k(X;\m) \lto \H_{n-k}(X;\m\otimes \Loc) \]
is an isomorphism\footnote{I.e.\ a PD space as described in the body of this paper.}. Let $\F$ be a spherical fibration over $X$ (of formal dimension $-n$) equipped with a map $c\colon \S \to r_!(\F) = \M\F$. Assume that its $\Z$-linearisation is equivalent to $\Loc[-n]$ such that under the Thom isomorphism
\[ \H_0(\M\F;\Z) \cong \H_n(X;\Loc) \]
the class $[c]$ is sent to the fundamental class $[X]$. Then the map $c \colon \S \to r_!(\F)$ exhibits $\F$ as the Spivak normal fibration of $X$.
\end{Prop}
\begin{proof}
We need to show that for all $X$-parametrised spectra $\g \colon X \to \Sp$, the induced map 
\[ r_*(\g) \lto r_!(\F \otimes \g) \]
is an equivalence. We observe that both functors commute with colimits and limits in $\g$; for the latter this is because $\F \otimes -$ is an equivalence since $\F$ is invertible. Moreover, since $\Fun(X,\Sp)$ is generated by objects of the form $i_!(\S)$ it suffices to check the equivalence for such objects. We note that in this case, both left and right hand side are bounded below spectra as they are given by a finite limit respectively colimits of pointwise bounded below objects. It therefore suffices to show that the map 
\[ r_*(\g) \otimes \Z \lto r_!(\F\otimes \g) \otimes \Z \]
is an equivalence in $\D(\Z)$, the derived $\infty$-category of $\Z$. We now use the commutative diagram (recall that $r_*$ is essentially a finite limit)
\[ \begin{tikzcd}	
	\Fun(X,\Sp) \ar[r,"r_*"] \ar[d,"-\otimes \Z"] & \Sp \ar[d,"-\otimes \Z"] \\
	\Fun(X,\D(\Z)) \ar[r,"r_*"] & \D(\Z) 
\end{tikzcd}\]
And see that it therefore suffices to prove that the map
\[ r_*(\g) \lto r_!((\F\otimes \Z)\otimes \g) \]
is an equivalence for all $\g \in \Fun(X,\D(\Z))$ and $r_*,r_!$ viewed as functors $\Fun(X,\D(\Z)) \to \D(\Z)$. Again, both of these functors commute with limits and colimits, so by the (pointwise) Whitehead and Postnikov towers of $\g$, it suffices to show that the map 
\[ r_*(\m) \lto r_!(\m\otimes \Loc[-n]) \]
is an equivalence for all local systems of abelian groups $\m$ on $X$. But $r_*(\m) = C^*(X;\m)$ and $r_!(\m\otimes \Loc[-n]) = C_*(X;\m \otimes \Loc)[-n]$ are cohomology and homology with local coefficients, and the induced map between them is given by cap product with the class corresponding to $[c]$ under the Thom isomorphism
\[ \H_0(\M\F;\Z) \cong \H_0(\M\Loc;\Z) \cong \H_n(X;\Loc).\]
By assumption, this class is $[X]$, so the map is indeed an equivalence.
\end{proof}

\begin{Cor}
Let $M$ be a a closed manifold. Then the geometric Pontryagin-Thom collapse map $c_M\colon \S \to \M\nu_M$ exhibits the stable normal bundle $\nu_M$ of $M$ as the dualising spectrum of $M$. 
\end{Cor}
\begin{proof}
The orientation local system of $M$ is determined by $w_1(TM)= w_1(\nu_M)$. Therefore, the $\Z$-linearisation of $\nu_M$ is equivalent to the orientation local system of $M$. Furthermore, the Pontryagin--Thom collapse map has geometric degree 1. 
Therefore, under the Thom isomorphism
\[ \H_0(\M\nu;\Z) \cong \H_n(M;\Orloc_M) \]
the class $[c_M]$ corresponds to a generator, which is a fundamental class for $[M]$. Thus, classical Poincar\'e duality for $M$ shows that $M$ is a Poincar\'e duality complex and that the underlying spherical fibration of $\nu_M$ is the Spivak normal fibration on $M$.
\end{proof}

\begin{Prop}\label{Prop:same-definitions}
A finite anima is a PD complex in the sense of \cref{Def:PD-space} if and only if it is one in the classical sense explained in \cref{Section:PD-spaces-classical}
\end{Prop}
\begin{proof}
It is easy to see that a finite PD space in the sense of \cref{Def:PD-space} is also one in the sense of the body of the paper: The orientation local system $\Loc$ is the $\Z$-linearisation of its Spivak normal fibration, and the fundamental class is obtained from the Pontryagin--Thom collapse map. This data satisfies the required properties by the same argument as given in \cref{Prop:recognition-principle-for-Spivak-fibration}.

Conversely, suppose $X$ is a finite anima and a PD space in the sense of the body of the paper. 
By \cref{Prop:recognition-principle-for-Spivak-fibration}, it remains to show that there \emph{exists} a spherical fibration $\F$ over $X$ and a map $\S \to r_!(\F)$ giving rise the fundamental class of $X$ under the Thom isomorphism.
This is precisely the classical existence statement for Spivak normal fibrations of Poincar\'e duality spaces due to Spivak \cite{Spivak}. He gives a concrete construction of it as follows. By assumption, $X$ can be represented as a finite CW-complex, and then also as a finite simplicial complex. Any such complex can be embedded into $\R^n$ for suitably large $n$. One can then thicken up the embedding to a closed codimension zero embedding and therefore obtain a manifold $N$ with boundary (if $X$ is a manifold, think of this as being a closed tubular neighbourhood of the embedded manifold; a space isomorphic to the disk bundle of the normal bundle). This thickening can be chosen to retract onto to given embedding of $X$. In particular, one can restrict this retraction to the boundary $\partial N$ of $N$, and obtain a map $\partial N \to X$. The theorem of Spivak is then that the (homotopy) fibres of this map are (have the homotopy type of) spheres if and only if $X$ satisfies Poincar\'e duality in the sense of \cref{Prop:recognition-principle-for-Spivak-fibration}. See also \cite[Section 5]{CLM} for the details.
\end{proof}

\subsection{A proof of \cref{Lemma:cell-structure-on-Thom-spectrum}}\label{proof-of-lemma}
We finish this appendix with a proof of \cref{Lemma:cell-structure-on-Thom-spectrum}, whose setup we recall here for convenience. To begin, we note that $\B
\G$ and $\B\gl_1(\S)$ are two notations for the same homotopy type. We let $n\geq 1$ and fix an element $\xi \in \pi_{n+1}(\B\gl_1(\S))$ viewed as a stable spherical fibration (of formal rank 0) over $S^{n+1}$ and aim to give a presentation of the Thom spectrum $\M\xi$. For this, we recall that the Thom spectrum is given by the colimit over the functor 
\[ \xi \colon S^{n+1} \lto \B\gl_1(\S) \subseteq \Sp.\]
Writing $S^{n+1}$ as the suspension of $S^n$, we obtain the following pushout describing the colimit of $\xi$ as follows:
\[\begin{tikzcd} 
	\colim\limits_{S^n} \xi_{|S^n} \ar[r] \ar[d] & \S \ar[d] \\
	\S \ar[r] & \colim\limits_{S^{n+1}} \xi = \M\xi
\end{tikzcd}\]
By construction, $\xi_{|S^n}$ is a constant functor, so the colimit evaluates to $S^n_+\otimes \S$. It follows that under the equivalence $\colim\limits_{S^n} \simeq S^n_+ \otimes \S$, the above pushout square is equivalent to the square
\[ \begin{tikzcd}
	S^n_+\otimes \S \ar[r,"\mathrm{pr}"] \ar[d] & \S \ar[d] \\
	\S \ar[r] & \M\xi
\end{tikzcd}\]
where the map $\mathrm{pr}$ is induced by the map $S^n \to \ast$. The left vertical map is, however, \emph{not} simply the projection, rather it corresponds under adjunction to a map $S^n_+ \to \map(\S,\S) = \S$ which happens to land in the subspace $\gl_1(\S)$ and is then induced by the clutching function $S^n \to \gl_1(\S)$ associated to the map $\xi \colon S^{n+1} \to \B\gl_1(\S)$. Therefore, under the isomorphism $\pi_{n+1}(\B\gl_1(\S)) \cong \pi_n(\gl_1(\S)) \cong \pi_n(\S)$, the map $\xi$ corresponds to a map $x\colon S^n \to \S$ and we obtain a pushout 
\[ \begin{tikzcd}
	S^n_+ \otimes \S \ar[r,"\mathrm{pr}"] \ar[d,"x_+"'] & \S \ar[d] \\
	\S \ar[r] & \M\xi
\end{tikzcd}\]
from which we obtain a cofibre sequence
\[ S^n \otimes \S = \S^n \stackrel{\cdot x}{\lto} \S \lto \M\xi \]
as claimed.

\bibliographystyle{amsalpha}
\bibliography{mybib}

\newcommand{\etalchar}[1]{$^{#1}$}
\providecommand{\bysame}{\leavevmode\hbox to3em{\hrulefill}\thinspace}
\providecommand{\MR}{\relax\ifhmode\unskip\space\fi MR }
\providecommand{\MRhref}[2]{%
  \href{http://www.ams.org/mathscinet-getitem?mr=#1}{#2}
}
\providecommand{\href}[2]{#2}
\begin{thebibliography}{CDH{\etalchar{+}}20}

\bibitem[BLW10]{BLW}
A.~Bartels, W.~L{\"u}ck, and S.~Weinberger, \emph{On hyperbolic groups with
  spheres as boundary}, J. Differential Geom. \textbf{86} (2010), no.~1, 1--16.

\bibitem[BM91]{BM}
M.~Bestvina and G.~Mess, \emph{The boundary of negatively curved groups}, J.
  Amer. Math. Soc. \textbf{4} (1991), no.~3, 469--481.

\bibitem[Bro72]{Browder}
W.~Browder, \emph{Surgery on simply-connected manifolds}, Springer-Verlag, New
  York-Heidelberg, 1972, Ergebnisse der Mathematik und ihrer Grenzgebiete, Band
  65.

\bibitem[CDH{\etalchar{+}}20]{CDHIII}
B.~Calm{\'e}s, E.~Dotto, Y.~Harpaz, F.~Hebestreit, M.~Land, K.~Moi, D.~Nardin,
  T.~Nikolaus, and W.~Steimle, \emph{Hermitian {K}-theory for stable
  $\infty$-categories {III}: {Grothendieck--Witt groups of rings}},
  arXiv:2009.07225 (2020).

\bibitem[Cer68]{Cerf}
J.~Cerf, \emph{Sur les diff{\'e}omorphismes de la sph{\`e}re de dimension trois
  {$(\Gamma _{4}=0)$}}, Lecture Notes in Mathematics, No. 53, Springer-Verlag,
  Berlin-New York, 1968.

\bibitem[CLM15]{CLM}
D.~Crowley, W.~L{\"u}ck, and T.~Macko, \emph{Surgery theory: Foundations},
  available at http://www.mat.savba.sk/~macko/ (2015).

\bibitem[CS19]{CS}
K.~Cesnavicius and P.~Scholze, \emph{Purity for flat cohomology},
  \href{https://arxiv.org/abs/1912.10932}{arXiv:1912.10932}, 2019.

\bibitem[EL83]{EL}
B.~Eckmann and P.~Linnell, \emph{Poincar{\'e} duality groups of dimension two.
  {II}}, Comment. Math. Helv. \textbf{58} (1983), no.~1, 111--114.

\bibitem[EM80]{EM}
B.~Eckmann and H.~M{\"u}ller, \emph{Poincar{\'e} duality groups of dimension
  two}, Comment. Math. Helv. \textbf{55} (1980), no.~4, 510--520.

\bibitem[FLW19]{Cannon-conjecture}
S.~Ferry, W.~L\"{u}ck, and S.~Weinberger, \emph{On the stable {C}annon
  {C}onjecture}, J. Topol. \textbf{12} (2019), no.~3, 799--832.

\bibitem[FP95]{FP}
S.~C. Ferry and E.~K. Pedersen, \emph{Epsilon surgery theory}, Novikov
  conjectures, index theorems and rigidity, {V}ol.\ 2 ({O}berwolfach, 1993),
  London Math. Soc. Lecture Note Ser., vol. 227, Cambridge Univ. Press,
  Cambridge, 1995, pp.~167--226.

\bibitem[GHN17]{GHN}
D.~Gepner, R.~Haugseng, and T.~Nikolaus, \emph{Lax colimits and free fibrations
  in {$\infty$}-categories}, Doc. Math. \textbf{22} (2017), 1225--1266.

\bibitem[GS65]{GitlerStasheff}
S.~Gitler and J.~D. Stasheff, \emph{The first exotic class of {$BF$}}, Topology
  \textbf{4} (1965), 257--266.

\bibitem[Ham19]{Hambleton}
I.~Hambleton, \emph{Orientable {$4$}-dimensional {P}oincar\'{e} complexes have
  reducible {S}pivak fibrations}, Proc. Amer. Math. Soc. \textbf{147} (2019),
  no.~7, 3177--3179.

\bibitem[Hil20]{Hillman}
J.~A. Hillman, \emph{{Poincar\'{e} duality in Dimension 3}}, The open book
  series, vol.~3, Mathematical Sciences Publishers, 2020.

\bibitem[Hil21]{Hillman2}
\bysame, \emph{{$PD_3$-complexes bound}}, arXiv:2109.09947v2 (2021).

\bibitem[HLN21]{HLN}
F.~Hebestreit, M.~Land, and T.~Nikolaus, \emph{On the homotopy type of
  {L}-spectra of the integers}, J. Topol. \textbf{14} (2021), no.~1, 183--214.

\bibitem[HM78]{HM}
I.~Hambleton and R.~J. Milgram, \emph{Poincar{\'e} transversality for double
  covers}, Canad. J. Math. \textbf{30} (1978), no.~6, 1319--1330.

\bibitem[HV93]{HV}
J.-C. Hausmann and P.~Vogel, \emph{Geometry on {P}oincar\'{e} spaces},
  Mathematical Notes, vol.~41, Princeton University Press, Princeton, NJ, 1993.

\bibitem[Jon73]{Jones}
L.~Jones, \emph{Patch spaces: a geometric representation for {P}oincar{\'e}
  spaces}, Ann. of Math. (2) \textbf{97} (1973), 306--343.

\bibitem[Kle01]{Klein}
J.~R. Klein, \emph{The dualizing spectrum of a topological group}, Math. Ann.
  \textbf{319} (2001), no.~3, 421--456.

\bibitem[KLPT17]{KLPT}
D.~Kasprowski, M.~Land, M.~Powell, and P.~Teichner, \emph{Stable classification
  of 4-manifolds with 3-manifold fundamental group}, Journal of Topology
  \textbf{10} (2017), no.~3.

\bibitem[KM63]{KM}
M.~A. Kervaire and J.~W. Milnor, \emph{Groups of homotopy spheres. {I}}, Ann.
  of Math. (2) \textbf{77} (1963), 504--537.

\bibitem[KS77]{KS}
R.~C. Kirby and L.~C. Siebenmann, \emph{Foundational essays on topological
  manifolds, smoothings, and triangulations}, Princeton University Press,
  Princeton, N.J.; University of Tokyo Press, Tokyo, 1977, With notes by John
  Milnor and Michael Atiyah, Annals of Mathematics Studies, No. 88.

\bibitem[KT01]{KT}
R.~C. Kirby and L.~R. Taylor, \emph{A survey of 4-manifolds through the eyes of
  surgery}, Surveys on surgery theory, {V}ol. 2, Ann. of Math. Stud., vol. 149,
  Princeton Univ. Press, Princeton, NJ, 2001, pp.~387--421.

\bibitem[Lev72]{Levitt}
N.~Levitt, \emph{Poincar{\'e} duality cobordism}, Ann. of Math. (2) \textbf{96}
  (1972), 211--244.

\bibitem[Mil88]{Milgram}
R.~J. Milgram, \emph{Some remarks on the {K}irby-{S}iebenmann class}, Algebraic
  topology and transformation groups ({G}\"{o}ttingen, 1987), Lecture Notes in
  Math., vol. 1361, Springer, Berlin, 1988, pp.~247--252.

\bibitem[Moi77]{Moise}
E.~E. Moise, \emph{Geometric topology in dimensions {$2$} and {$3$}},
  Springer-Verlag, New York-Heidelberg, 1977, Graduate Texts in Mathematics,
  Vol. 47.

\bibitem[Qui72]{Quinn}
Frank Quinn, \emph{Surgery on {P}oincar{\'e} and normal spaces}, Bull. Amer.
  Math. Soc. \textbf{78} (1972), 262--267.

\bibitem[Ran79]{RanickiTSO}
A.~A. Ranicki, \emph{The total surgery obstruction}, Algebraic topology,
  {A}arhus 1978 ({P}roc. {S}ympos., {U}niv. {A}arhus, {A}arhus, 1978), Lecture
  Notes in Math., vol. 763, Springer, Berlin, 1979, pp.~275--316.

\bibitem[Ran92]{Ranicki}
\bysame, \emph{Algebraic {$L$}-theory and topological manifolds}, Cambridge
  Tracts in Mathematics, vol. 102, Cambridge University Press, Cambridge, 1992.

\bibitem[Spi67]{Spivak}
M.~Spivak, \emph{Spaces satisfying {P}oincar{\'e} duality}, Topology \textbf{6}
  (1967), 77--101.

\bibitem[Tei93]{Teichner}
P.~Teichner, \emph{On the signature of four-manifolds with universal covering
  spin}, Math. Ann. \textbf{295} (1993), no.~4, 745--759.

\bibitem[Wal67]{Wall}
C.~T.~C. Wall, \emph{Poincar{\'e} complexes. {I}}, Ann. of Math. (2)
  \textbf{86} (1967), 213--245.

\bibitem[Wei92]{Weiss}
M.~Weiss, \emph{Visible {$L$}-theory}, Forum Math. \textbf{4} (1992), no.~5,
  465--498.

\end{thebibliography}

\end{document}